\documentclass[11pt]{article}

\usepackage{setspace}
\usepackage{amsmath}
\usepackage{amsfonts}
\usepackage{amssymb}
\usepackage{graphicx}
\usepackage{amsthm}
\usepackage{enumerate}
\usepackage{color}
\usepackage{MnSymbol}
\usepackage{fullpage}
\usepackage{slashbox}
\usepackage{float}

\newtheorem{theorem}{Theorem}
\newtheorem{lemma}[theorem]{Lemma}

\newtheorem{proposition}[theorem]{Proposition}
\newtheorem{corollary}[theorem]{Corollary}

\title{Multiplication groups and inner mapping groups\\of Cayley--Dickson loops}
\author{Jenya Kirshtein \thanks{ykirshte@du.edu, Department of Mathematics, University of Denver, 2360 South Gaylord street, Denver, CO 80208, USA}}
\date{\vspace{-5ex}}

\begin{document}

\maketitle

\begin{abstract}
The Cayley--Dickson loop $Q_n$ is the multiplicative closure of basic elements of the algebra constructed by $n$ applications of the Cayley--Dickson doubling process (the first few examples of such algebras are real numbers, complex numbers, quaternions, octonions, sedenions). We establish that the inner mapping group $Inn(Q_n)$ is an elementary abelian $2$-group of order $2^{2^n-2}$ and describe the multiplication group $Mlt(Q_n)$ as a semidirect product of $Inn(Q_n)\times \mathbb{Z}_2$ and an elementary abelian $2$-group of order $2^n$. We prove that one-sided inner mapping groups $Inn_l(Q_n)$ and $Inn_r(Q_n)$ are equal, elementary abelian $2$-groups of order $2^{2^{n-1}-1}$. We establish that one-sided multiplication groups $Mlt_l(Q_n)$ and $Mlt_r(Q_n)$ are isomorphic, and show that $Mlt_l(Q_n)$ is a semidirect product of $Inn_l(Q_n)\times\mathbb{Z}_2$ and an elementary abelian $2$-group of order~$2^n$. 

\end{abstract}
\let\thefootnote\relax 
\let\thefootnote\relax

\section{Introduction}\label{sec:intro}
\subsection{Cayley--Dickson loops}\label{subsec:CDL}
The \emph{Cayley--Dickson doubling process} produces a sequence of power-associative algebras over a field~$F$ (see~\cite{SprVel:00}).
Let $\mathbb{A}_{0}=F$ with conjugation $a^{*}=a$ for all $a\in F$. Let $\mathbb{A}_{n+1}=\{(a,b)\left|\right.a,b\in A_{n}\}$ for $n\in \mathbb{N}$, where multiplication, addition, and conjugation are defined as follows:
\begin{eqnarray*}
\label{eqn:mult} (a,b)(c,d) & = & (ac-d^{*}b,da+bc^{*}), \\ 
\label{eqn:add} (a,b)+(c,d) & = & (a+c,b+d), \\ 
\label{eqn:conj} (a,b)^{\ast} & = & (a^{\ast},-b). 
\end{eqnarray*}
The dimension of $\mathbb{A}_{n}$ over $F$ is~$2^{n}$. 

We consider multiplicative structures that arise from the Cayley--Dickson process. A \emph{loop} is a nonempty set $Q$ with a binary operation such that there is a neutral element $1\in Q$ satisfying $1\cdot x=x\cdot 1=x$ for all $x\in Q$, and for every $a,b\in Q$ there is a unique $x$ and a unique $y$ satisfying $x\cdot a=b$, $a\cdot y=b$. Define \emph{Cayley--Dickson loops} $\left(Q_{n},\cdot \right)$ over $F$ inductively as follows: 
\begin{equation*}\label{eqn:CD2}
 Q_{0}=\{1,-1\},\ \  Q_{n}=\{(x,0),(x,1)\left|\right.x\in Q_{n-1}\},
\end{equation*}
with multiplication
\begin{eqnarray*} 
\label{eqn:mult0} (x,0)(y,0) & = & (xy,0),\\
\label{eqn:mult1} (x,0)(y,1) & = & (yx,1), \\
\label{eqn:mult2} (x,1)(y,0) & = & (xy^{*},1),\\
\label{eqn:mult3} (x,1)(y,1) & = & (-y^{*}x,0), 
\end{eqnarray*}
and conjugation
\begin{eqnarray*} \label{eqn:conj1}
(x,0)^{*} & = & (x^{*},0),\\ 
(x,1)^{*} & = & (-x,1). 
\end{eqnarray*}

Note that the Cayley--Dickson loops are independent of the underlying field $F$ of characteristic not two. The reader can assume $F=\mathbb{R}$ without loss of generality from now on. The first few Cayley--Dickson loops are the real group $\mathbb{R}_{2}$ (abelian); the complex group $\mathbb{C}_{4}$ (abelian); the quaternion group $\mathbb{H}_{8}$ (not abelian); the octonion loop $\mathbb{O}_{16}$ (Moufang); the sedenion loop $\mathbb{S}_{32}$ (not Moufang); the loop $\mathbb{T}_{64}$. Note that the Cayley--Dickson loops are not associative after dimension $8$. The study of basic elements provides information about the underlying algebra, and is of interest, for example, in Lie theory, graph theory, quantum physics, functional analysis (see~\cite{Baez:02},~\cite{Ludkovsky:10}).

The order of $Q_{n}$ is $2^{n+1}$. The loop $Q_{n}$ embeds into $Q_{n+1}$ by $x\mapsto (x,0)$, so that
\[
 Q_{n}\cong\left\{(x,0)\left|\right.(x,0)\in Q_{n+1}\right\}.
\] 
All elements of $Q_{n}$ have norm one. 
Denote the opposite of an element $(x_1,x_2,x_3,\ldots,x_{n+1})$ by
\[
-(x_1,x_2,x_3,\ldots,x_{n+1})=(-x_1,x_2,x_3,\ldots,x_{n+1}).
\]
The elements $1_{Q_n},-1_{Q_n}\in Q_n$ are
\begin{eqnarray*}
1_{Q_n}&=&(1,\underbrace{0,\ldots,0}_{n}),\\
-1_{Q_n}&=&(-1,\underbrace{0,\ldots,0}_{n}).
\end{eqnarray*}
We call $1_{Q_n}$ by $1$, and $-1_{Q_n}$ by $-1$. One can see that $1$ and $-1$ commute and associate with every element of $Q_n$. 
We denote the loop generated by elements $x_{1},\ldots,x_{n}$ of a loop $L$ by $\left\langle x_{1},\ldots,x_{n}\right\rangle$.
Let~$i_{n}$ be the element $(1_{Q_{n-1}},1)=(1,\underbrace{0,\ldots,0}_{n-1},1)$ of $Q_{n}$. Such element $i_n$ satisfies $Q_{n}=Q_{n-1}\cup (Q_{n-1}i_{n})=\left\langle Q_{n-1},i_n\right\rangle$. Thus $Q_{n}=\left\langle i_{1},i_{2},\ldots, i_{n}\right\rangle$. We call $i_{1},i_{2},\ldots, i_{n}$ the \emph{canonical generators} of~$Q_{n}$. Any~$x\in Q_{n}$ can be written as 
\begin{equation*}\label{eqn:x}
x=\pm\prod_{j=1}^n i_{j}^{\epsilon_{j}},\ \ \epsilon_{j}\in \{0,1\}.
\end{equation*}
For example,   
\begin{eqnarray*}
 Q_0=\mathbb{R}_{2} & = & \{1,-1\},\\
 Q_1=\mathbb{C}_{4} & = & \pm \{(1,0),(1,1)\}=\left\langle  i_1\right\rangle = \{1,-1,i_1,-i_1\},\\
 Q_2=\mathbb{H}_{8} & = & \pm \{(1,0,0),(1,1,0),(1,0,1),(1,1,1)\} = \left\langle i_1,i_2\right\rangle = \pm \{1,i_{1},i_{2},i_{1}i_{2}\},\\
 Q_3=\mathbb{O}_{16}& = & \left\langle i_1,i_2,i_3\right\rangle = \pm\{1,i_{1},i_{2},i_1i_2,i_3,i_1i_3,i_2i_3,i_1i_2i_3\},\\
 \end{eqnarray*}
We use 
\[
e=i_n
\]
for the unit added in the last step of the process. Let us recall some basic properties of the Cayley--Dickson loops, for more details see~\cite{Kirshtein:12}.

\begin{proposition}\label{prop:prop}
Let $Q_n$ be a Cayley--Dickson loop, and $x,y\in Q_n$. The following hold:
\begin{enumerate}
	\item the conjugate $x^{*}=-x$ for $x\neq\pm1$, $1^{*}=1$, $(-1)^{*}=-1$;
	\item the inverse $x^{-1}=x^{*}$;
	\item the order $\left|x\right|=4$ for $x\neq \pm1$, $\left|1\right|=1$, $\left|-1\right|=2$;
	\item the loop $Q_n$ is Hamiltonian (every subloop $S$ is normal in $Q_n$, i.e., $xS=Sx$, $(xS)y=x(Sy)$, $x(yS)=(xy)S$).
\end{enumerate}
\end{proposition}

A loop~$Q$ is \emph{diassociative} if every pair of elements of $Q$ generates a group in~$Q$. A loop $Q$ is the \emph{inverse property loop} if there exist bijections $\lambda:x\mapsto x^\lambda$ and $\rho:x\mapsto x^\rho$ on $Q$ such that $x^\lambda(xy)=y$ and $(yx)x^\rho=y$ for every $y\in Q$. If $Q$ is an inverse property loop, then it satisfies the \emph{anti-automorphic inverse property} $(xy)^{-1}=y^{-1}x^{-1}$ for every $x,y\in Q$. One can see that diassociative loops are also inverse property loops. 

\begin{theorem}\cite{Culbert:07}
Cayley--Dickson loop  $Q_n$ is diassociative. Any pair of elements of $Q_n$ generate a subgroup of the quaternion group. In particular, a pair $x,y$ generates a real group when $x=\pm1$ and $y=\pm1$; a complex group when either $x=\pm1$, or $y=\pm1$ (but not both), or $x=\pm y\neq\pm 1$; a quaternion group otherwise. 
\end{theorem}

For a loop $Q$ and $x, y, z \in Q$ define the \emph{commutator} $[x,y]$ by $xy=(yx)[x,y]$ and the \emph{associator} $[x,y,z]$ by $xy\cdot z=(x\cdot yz)[x,y,z]$. The \emph{center} of a loop $Q$, denoted by $Z(Q)$, is the set of elements that commute and associate with every element of $Q$, more precisely, $Z(Q)=\{a\in Q \left| \right. ax=xa, a\cdot xy=ax\cdot y, xa\cdot y=x\cdot ay, xy\cdot a=x\cdot ya, \forall x,y\in Q\}$.

\begin{lemma}\label{lemma:comm-assoc}
Let $Q_n$ be a Cayley--Dickson loop, and $x,y,z\in Q_n$. The following hold:
\begin{enumerate}
	\item the commutator $[x,y]\in \left\{1,-1\right\}$, in particular, $[x,y]=-1$ when $\left\langle x,y\right\rangle\cong \mathbb{H}_8$, and $[x,y]=1$ when $\left\langle x,y\right\rangle < \mathbb{H}_8$;
	\item the associator $[x,y,z]\in \left\{1,-1\right\}$, in particular, $[x,y,z]=1$ when $\left\langle x,y,z\right\rangle\leq \mathbb{H}_{8}$, and $[x,y,z]=-1$ when $\left\langle x,y,z\right\rangle\cong \mathbb{O}_{16}$;
	\item the center $Z(Q_n)=\left\{1,-1\right\}$ when $n\geq 2$, and $Z(Q_n)=Q_n$ when $n<2$.
\end{enumerate}
\end{lemma}

\begin{theorem}\label{cor:Zn}
If $Q_n$ is a Cayley--Dickson loop, then $Q_{n}/\{1,-1\}\cong (\mathbb{Z}_{2})^{n}$. 
\end{theorem}

\begin{lemma}\label{lemma:order}
Let $S$ be a subloop of $Q_{n}$. The following hold:
\begin{enumerate}
	\item the center $Z(Q_n)\leq S$ for any $S\leq Q_n, S\neq \{1\}, n\geq 2$;
	\item \label{lemma:order1} If $x\in  Q_{n}\backslash S$, then $\left|\left\langle S,x\right\rangle\right|=2\left|S\right|$ when $S\neq\{1\}$, and $\left\langle S,x\right\rangle=\left\{1,-1,x,-x\right\}$ when $S=\{1\}$; 
	\item \label{lemma:order3} any $n$ elements of a Cayley--Dickson loop generate a subloop of order $2^k,$ $k\leq n+1$; 
	\item the order of $S$ is $2^{m}$ for some $m\leq n$.
\end{enumerate}
\end{lemma}

It follows from Lemma~\ref{lemma:order} that $\left|\left\langle x,y,z\right\rangle\right|\leq 16$ for $x,y,z\in Q_n$. Subloops of size~$16$ of the sedenion loop $\mathbb{S}_{32}$ are either isomorphic to the octonion loop $\mathbb{O}_{16}$, or the quasioctonion loop $\tilde{\mathbb{O}}_{16}$ (see~\cite{Cawagas:04}). In fact, a stronger statement holds.

\begin{lemma}\label{lemma:quasioct}
Let $Q_n$ be a Cayley--Dickson loop. If $x,y,z$ are elements of $Q_{n}$ such that $\left|\left\langle x,y,z\right\rangle\right|=16$, then either
\begin{equation}
\nonumber \left\langle x,y,z\right\rangle \cong \mathbb{O}_{16} \mbox{ or } \left\langle x,y,z\right\rangle \cong \tilde{\mathbb{O}}_{16}.
\end{equation}
\end{lemma}

Any subloop of size~$16$ of~$Q_n$ containing the element $e$ is isomorphic to the octonion loop $\mathbb{O}_{16}$.

\begin{lemma}\label{lemma:aut2}
Let $Q_n$ be a Cayley--Dickson loop. If $x$ and $y$ are elements of $Q_{n}$ such that $e\notin \left\langle x,y\right\rangle\cong \mathbb{H}_{8}$, then $\left\langle x,y,e\right\rangle\cong\mathbb{O}_{16}$.
\end{lemma}

Finally, for a nonassociative loop $Q_n$ the automorphism group $Aut(Q_n)$ is a direct product of $(\mathbb{Z}_2)^{n-3}$ and $Aut(\mathbb{O}_{16})$ (it was established in~\cite{KocaKoc:1995} that $Aut(\mathbb{O}_{16})$ has size $1344$ and is an extension of the elementary abelian group $(\mathbb{Z}_{2})^3$ by the simple group of symmetries of the Fano plane $PSL_{2}(7)$).

\begin{theorem}\label{thm:autdirectproduct}
Let $Q_n$ be a Cayley--Dickson loop and let $n\geq 3$. Then $Aut\left(Q_{n}\right)\cong Aut(\mathbb{O}_{16})\times (\mathbb{Z}_2)^{n-3}$. The order of $Aut(Q_{n})$ is therefore $1344\cdot 2^{n-3}$. 
\end{theorem}

\subsection{Inner mapping groups and multiplication groups}\label{subsec:innmlt}
Let us recall some basic facts about inner mapping groups and multiplication groups, notions that are of significant interest and importance in loop theory. Let $Q$ be a loop and $x,a\in Q$. Mappings $L_x(a)=xa$ and $R_x(a)=ax$ are called \emph{left} and \emph{right translations}, these mappings are permutations on~$Q$. Define the following subgroups of~$Sym(Q)$:
\begin{eqnarray}
\nonumber &&\mbox{\emph{the multiplication group of $Q$}, } Mlt(Q)=\left\langle L_{x},R_{x}\left| \right.x\in Q \right\rangle,\\
\nonumber &&\mbox{\emph{the inner mapping group of $Q$}, } Inn(Q)=Mlt(Q)_{1}=\{f\in Mlt(Q)\left| \right.f(1)=1\},\\ 
\nonumber &&\mbox{\emph{the left multiplication group of $Q$}, } Mlt_l(Q)=\left\langle L_{x}\left| \right.x\in Q \right\rangle,\\
\nonumber &&\mbox{\emph{the left inner mapping group of $Q$}, } Inn_l(Q)=Mlt_l(Q)_{1}=\{f\in Mlt_l(Q)\left| \right.f(1)=1\},\\ 
\nonumber &&\mbox{\emph{the right multiplication group of $Q$}, } Mlt_r(Q)=\left\langle R_{x}\left| \right.x\in Q \right\rangle,\\
\nonumber &&\mbox{\emph{the right inner mapping group of $Q$}, } Inn_r(Q)=Mlt_r(Q)_{1}=\{f\in Mlt_r(Q)\left| \right.f(1)=1\}. 
\end{eqnarray}

Let $R_Q=\{R_x\left| \right. x \in Q\}$. Then $R_Q$ is a \emph{left transversal} to $Inn(Q)$ in $Mlt(Q)$, and also a \emph{right transversal} to $Inn(Q)$ in $Mlt(Q)$. That is, for every $f\in Mlt(Q)$ there is a unique $x\in Q$ and a unique $y\in Q$ such that $f\in R_xInn(Q)$, $f\in Inn(Q)R_y$. An analogous statement is true for $L_Q=\{L_x\left| \right. x \in Q\}$. 

Define the \emph{middle, left} and \emph{right inner mappings} on $Q$ by $T_x=L^{-1}_{x}R_x$, $L_{x,y}=L^{-1}_{yx}L_yL_x$, and $R_{x,y}=R^{-1}_{xy}R_yR_x$. 
Note that the inner mapping $T_x$ plays the role of conjugation, and the mappings $L_{x,y}$, $R_{x,y}$ measure deviations from associativity, just as $T_x$ measures deviations from commutativity. In an inverse property loop we have $R^{-1}_x=R_{x^{-1}}$ and $L^{-1}_x=L_{x^{-1}}$.

\begin{theorem}\label{thm:inn-gen}\cite{Pflugfelder:90}
Let $Q$ be a loop. Then 
\begin{eqnarray}
\nonumber Inn(Q)&=&\left\langle L_{x,y}, R_{x,y}, T_x\left|\right.x,y\in Q \right\rangle,\\
\nonumber Inn_l(Q)&=&\left\langle L_{x,y}\left|\right.x,y\in Q \right\rangle,\\
\nonumber Inn_r(Q)&=&\left\langle R_{x,y}\left|\right.x,y\in Q \right\rangle.
\end{eqnarray}
\end{theorem}

\begin{lemma}\label{lemma:mltsize}
Let $Q$ be a finite loop. Then $\left|Mlt(Q)\right|=\left|Q\left|\right|Inn(Q)\right|$, $\left|Mlt_l(Q)\right|=\left|Q\left|\right|Inn_l(Q)\right|$, and $\left|Mlt_r(Q)\right|=\left|Q\left|\right|Inn_r(Q)\right|$.
\end{lemma}

\begin{lemma}\label{rmk:mlt_g}
Let $G$ be a group. If $G$ is abelian, then $Mlt(G)\cong G$ and $Inn(G)\cong\{1\}$. If $G$ is not abelian, then $\nonumber Mlt(G)\cong (G\times G)\slash \{(g,g)\left|\right.g\in Z(G)\}$ and $Inn(G)\cong G\slash Z(G)$.
\end{lemma}

\subsection{Overview}\label{subsec:overview}
%
%

For a Cayley--Dickson loop $Q_n$ we study inner mapping groups $Inn(Q_n)$ and multiplication groups $Mlt(Q_n)$. This work is organized as follows. In Section~\ref{sec:inn} we establish that elements of $Mlt(Q_n)$ are even permutations and have order $1,2$ or~$4$. We prove that $Inn(Q_n)$ is an elementary abelian $2$-group of order $2^{2^n-2}$, moreover, every $f \in Inn(Q_n)$ is a product of disjoint transpositions of the form $(x, -x)$. This implies that nonassociative Cayley--Dickson loops are not automorphic. We begin Section~\ref{sec:mult} by proving a number of lemmas that allow to construct an elementary abelian $2$-group $K$ based on left translations by the canonical generators of $Q_n$. We then show that the direct product of $Inn(Q_n)$ and a cyclic group $\mathbb{Z}_2$ is a normal subgroup of $Mlt(Q_n)$, and that $Mlt(Q_n)$ is a semidirect product of $Inn(Q_n)\times\mathbb{Z}_2$ and $K$. 
In Section~\ref{sec:rlinn} we prove that the groups $Inn_l(Q_n)$ and $Inn_r(Q_n)$ are equal, elementary abelian $2$-groups of order $2^{2^{n-1}-1}$, and establish that the groups $Mlt_l(Q_n)$ and $Mlt_r(Q_n)$ are isomorphic. We conclude with Section~\ref{sec:rlmlt} showing that $Mlt_l(Q_n)$ is a semidirect product of $Inn_l(Q_n)\times\mathbb{Z}_2$ and $K$.
We used GAP system for computational discrete algebra~\cite{GAP4}, specifically the LOOPS package~\cite{NagyVojt:06}, to perform numerical experiments and verify conjectures, however, the final results do not rely on computational proofs. 
\section{Inner Mapping Groups}\label{sec:inn}

In this section we discuss inner mapping groups and begin to study multiplication groups of the Cayley--Dickson loops $Q_n$. When $n\leq 2$, the loop $Q_n$ is a group, and the structure of $Mlt(Q_n)$ and $Inn(Q_n)$ is known (see Lemma~\ref{rmk:mlt_g}). We therefore focus on nonassociative Cayley--Dickson loops~$Q_n$,~$n\geq 3$. 

\begin{lemma}\label{lemma:mlt-even}
Let $Q_n$ be a Cayley--Dickson loop. Elements of $Mlt(Q_n)$ are even permutations.
\end{lemma}
\begin{proof}
Consider $L_x$. If $|x|=1$ then $L_x = id$. If $|x|=2$ then $L_xL_x(y) = xxy = y$ for every $y$, so $L_x$ is a product of $\left|Q_n\right|/2=2^{n}$ transpositions (of the form $(y, xy)$), and since $2^{n}$ is even, $L_x$ is even. If $|x|=4$ then $L_x$ is a product of $2^{n-1}$ $4$-cycles (of the form $(y, xy, xxy, xxxy)$), and since $2^{n-1}$ is even, $L_x$ is even. Similarly for right translations. Hence $Mlt(Q_n)$ is generated by even permutations, and it therefore consists of even permutations.
\end{proof}
If $\left\langle x,y,z\right\rangle\leq \mathbb{H}_8$, then $\left\langle x,y,z\right\rangle$ is a group and $[y,x,z]=\left[z,x,y\right]=1$. If $\left|\left\langle x,y,z\right\rangle\right|=16$, then 
\[
\left\langle x,y,z\right\rangle=\pm \{1,x,y,xy,z,xz,yz,(xy)z\},
\]
where all elements are distinct. This implies that $z\notin \pm \{1,xy\}$, $x\notin \pm \{1,y\}$, $x\notin \pm \{1,zy\}$, $y\notin \pm \{1,z\}$, and by Lemma~\ref{lemma:comm-assoc}
\begin{equation}\label{eqn:commsign}
[xy,z] = [x,y] = [x,zy] = [y,z] = -1.
\end{equation}

\begin{lemma}\label{1}
Let $x,y,z$ be elements of a Cayley--Dickson loop $Q_n$, then
\[   [x,y,z]=[z,y,x]. \]
\end{lemma}
\begin{proof}
We have
\begin{eqnarray*}
    xy\cdot z &=& [xy,z]z\cdot xy = [xy,z][x,y]z\cdot yx = [xy,z][x,y][z,y,x]zy\cdot x\\
     &=& [xy,z][x,y][z,y,x][x,zy]x\cdot zy = [xy,z][x,y][z,y,x][x,zy][y,z]x\cdot yz\\
     &=& [xy,z][x,y][z,y,x][x,zy][y,z][x,y,z]xy\cdot z.
\end{eqnarray*}
If $\langle x,y,z\rangle$ is a group then we are done, else $[xy,z] = [x,y] = [x,zy] = [y,z] = -1$ by \eqref{eqn:commsign} and we are done again.
\end{proof}

\begin{lemma}\label{lemma:Tx}
Let $Q_n$ be a Cayley--Dickson loop, and let $x,y\neq\pm 1$, $x\neq \pm y$ be elements of $Q_n$. Then
\begin{eqnarray}
\nonumber \label{eqn:Tx1} T_{x}&=&\prod_{1,x\neq z\in Q_n\slash\{\pm1\}}(z,-z),\\
\label{eqn:Tx3} T_{y}T_{x}&=&(x,-x)(y,-y).\\
\nonumber \label{eqn:Lx1} L_{x,e}&=&\prod_{1,x,e,xe\neq z\in Q_n\slash\{\pm1\}}(z,-z),\\
\label{eqn:Rx3} L_{y,e}L_{x,e}&=&(x,-x)(y,-y)(xe,-xe)(ye,-ye),\ \ \mbox{for } x,y\neq\pm e.
\end{eqnarray}
\end{lemma}
\begin{proof}
Consider $T_x, R_{x,y}, L_{x,y}$ acting on $z\in Q_n$. Using diassociativity, 
\begin{eqnarray}
\label{gen1} T_{x}(z)&=&x^{-1}(zx)=[x,z]x^{-1}(xz)=[x,z](x^{-1}x)z=[x,z]z,\\
\label{gen3} L_{x,y}(z)&=&(yx)^{-1}(y(xz))=[y,x,z](yx)^{-1}((yx)z)\\
\nonumber 						 &=&[y,x,z]((yx)^{-1}(yx))z=[y,x,z]z,\\
\label{gen2} R_{x,y}(z)&=&((zx)y)(xy)^{-1}=[z,x,y](z(xy))(xy)^{-1}\\
\nonumber							 &=&[z,x,y]z((xy)(xy)^{-1})=[z,x,y]z.
\end{eqnarray}
Let $x,y\neq \pm 1$, $x\neq \pm y$. If $z\in \pm \{1,x\}$, then $\left\langle x,z\right\rangle\cong \left\langle x\right\rangle \cong \mathbb{C}_4$, and $[x,z]=1$. Otherwise, $\left\langle x,z\right\rangle\cong \mathbb{H}_{8}$, and $[x,z]=-1$. Using \eqref{gen1}, 
\begin{equation*}
\label{genT} T_{x}(z)=[x,z]z=\begin{cases}
z, &\mbox{if }z\in \pm \{1,x\},\\
-z &\mbox{otherwise.}
\end{cases}
\end{equation*}
Similarly, if $z\in \pm \{x,y\}$, then $[y,z][x,z]=-1$. Otherwise, if $z\neq \pm 1$, then $\left\langle x,z\right\rangle\cong \left\langle y,z\right\rangle\cong \mathbb{H}_{8}$, and $[y,z]=[x,z]=-1$, if $z= \pm 1$, then $\left\langle x,z\right\rangle\cong \left\langle y,z\right\rangle\cong \mathbb{C}_{4}$, and $[y,z]=[x,z]=1$. We get
\begin{equation*}\label{eqn:inn_flip1} 
T_{y}T_{x}(z)=[y,z][x,z]z=\begin{cases}
-z, &\mbox{if }z\in \pm \{x,y\},\\
z &\mbox{otherwise.}
\end{cases}
\end{equation*}   
Let $x,y\neq \pm e$. If $z\in \pm \{1,x,e,xe\}$, then $\left\langle e,x,z\right\rangle\cong \left\langle e,x\right\rangle \cong \mathbb{H}_8$, and $[e,x,z]=1$. Otherwise, $\left\langle e,x,z\right\rangle\cong \mathbb{O}_{16}$ by Lemma~\ref{lemma:aut2}, and $[e,x,z]=-1$. Using \eqref{gen3},
\begin{equation*}
\label{genL} L_{x,e}(z)=[e,x,z]z=
\begin{cases}
z, &\mbox{if }z\in \pm \{1,x,e,xe\},\\
-z &\mbox{otherwise.}
\end{cases}
\end{equation*}
Similarly,
\begin{equation*}\label{eqn:rinn_flip1} 
L_{y,e}L_{x,e}(z)=[e,y,z][e,x,z]z=
\begin{cases}
-z, &\mbox{if }z\in \pm \{x,y,xe,ye\},\\
z &\mbox{otherwise.}
\end{cases}
\qedhere
\end{equation*}   
\end{proof}
\begin{corollary}\label{cor:RLmaps}
Let $Q_n$ be a Cayley--Dickson loop. Then 
\[
L_{x,y}=R_{x,y}\mbox{ for all }x,y\in Q_n.
\]
\end{corollary}
\begin{proof}
Let $x,y,z\in Q_n$. By Lemma~\ref{1},
\begin{equation}
\nonumber \left[y,x,z\right]=[z,x,y],
\end{equation}
$L_{x,y}=R_{x,y}$ follows from \eqref{gen3}, \eqref{gen2} in Lemma~\ref{lemma:Tx}.
\end{proof}
\begin{theorem}\label{thm:inn-invol}
Let $Q_n$ be a Cayley--Dickson loop, $n\geq1$. Then $Inn(Q_n)$ is an elementary abelian $2$-group of order $2^{2^n-2}$. Moreover, every $f\in Inn(Q_n)$ is a product of disjoint transpositions of the form $(x,-x)$.
\end{theorem}
\begin{proof}
Recall that $Z(Q_n) = \{1,-1\}$. Inner mappings fix $Z(Q_n)$ pointwise, therefore 
\[
f(1) = 1, f(-1) = -1.
\]
Let $x\in Q_n, x\neq \pm 1$. Then $|x| = 4$ and $S = \left\langle x\right\rangle = \{1, x, -1, -x\}$. We know that $Q_n$ is Hamiltonian, therefore $S\trianglelefteq Q_n$. Inner mappings fix normal subloops, thus $f(S) = S$, and it follows that either $f(x) = x$, $f(-x) = -x$, or $f(x) = -x$, $f(-x) = x$. 
Hence every $f$ has the desired form. In particular, $|f| = 2$.
A group of exponent $2$ is an elementary abelian $2$-group.\\
Let $e=i_{n}$ be a canonical generator of $Q_n$, let $x\in Q_n$, $x\notin \pm\left\{1,e\right\}$. Then $T_{x}T_{e}=(x,-x)(e,-e)$ by Lemma~\ref{lemma:Tx}. For every $f\in Inn(Q_n)$, there is $\tilde{f}=T_{x}T_{e}f\in Inn(Q_n)$ such that 
\begin{equation*}
\tilde{f}(z)=
\begin{cases}
-f(z), &\mbox{when }z\in \pm\{x,e\},\\
f(z), &\mbox{otherwise}.
\end{cases}
\end{equation*}
Also, the values of $f(e),f(-e)$ are uniquely determined by the values of $f(z),z\neq\pm e$, since $f$ should remain an even permutation by Lemma~\ref{lemma:mlt-even} (see Figure~\ref{fig:inn}). \\
\begin{figure}[H]
	\centering
		\includegraphics[width=0.6\textwidth]{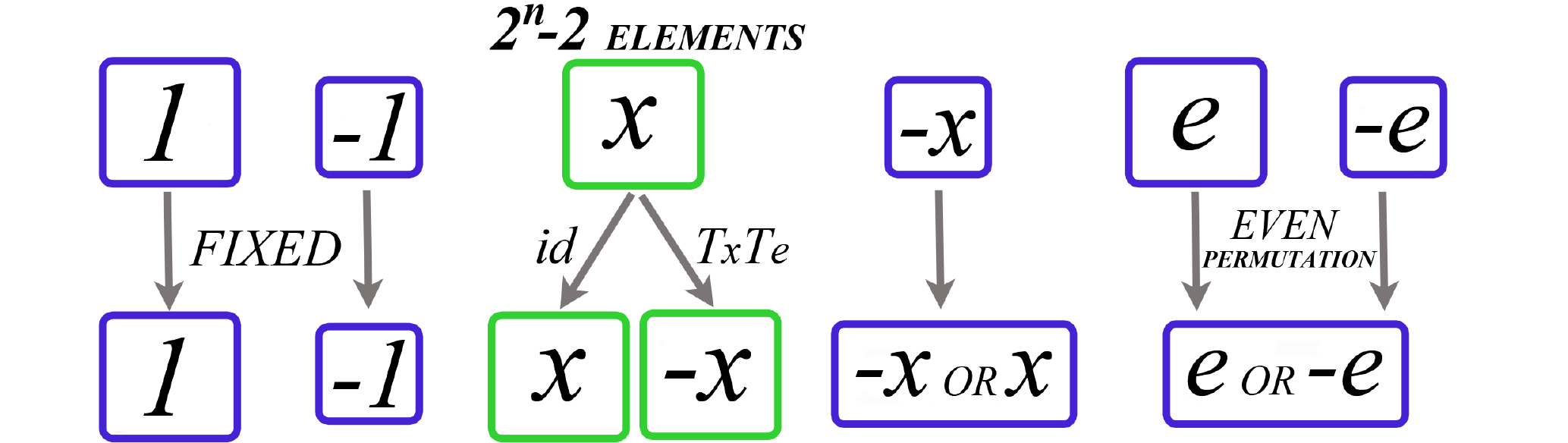}
	\caption{Inner mapping group of $Q_n$}
	\label{fig:inn}
\end{figure}

It follows that 
\begin{equation*}
\left|Inn(Q_n)\right|=2^{\left|Q_n\right|/2-2}=2^{2^n-2}.
\qedhere
\end{equation*}
\end{proof}
\begin{lemma}\label{lemma:inn-cycles}
Let $f\in Mlt(Q_n)$, then $\left|f\right|\in\{1,2,4\}$. In particular, $f$ is a product of disjoint $2$-cycles and $4$-cycles.
\end{lemma}
\begin{proof}
Denote by $-1$ the translation $L_{-1} = R_{-1} \in Mlt(Q_n)$. Let $f \in Mlt(Q_n)$. Let
$x \in Q_n$ be such that $f(1) = x$. Then there is $h \in Inn(Q_n)$ such that $f = L_{x}h$. If $x = 1$
then $f \in Inn(Q_n)$, and we are done. If $x = -1$ then $f = -h$, $f^{2} = (-h)(-h) = h^{2} =
1$, and we are also done. Assume that $x \neq\pm 1$. We know that $f\in L_x Inn(Q_n)$. There is $y \in Q_n$ such that $f \in Inn(Q_n)L_{y}$, we want to determine $y$. Let $f=L_x h=k L_y$ for some $h,k\in Inn(Q_n)$. Since $x\neq \pm 1$, we have $y\neq \pm 1$. Then $f(-y)=k L_y(-y)=k(1)=1$ and $f(-y)=L_x h(-y)=x (\pm y)$ (since $h(-y)$ is either $y$ or $-y$), so we conclude $y=x$ or $y=-x$. In the former case, we have $f = L_{x}h = kL_{x}$, and
so $f^{2} = kL_{x}L_{x}h = k(-1)h = -kh$, which has order at most two, so $f^{4} = 1$.
In the latter case, we have $f = L_{x}h = kL_{-x}$, and so
$f^{2} = kL_{-x}L_{x}h = kh$, which has order at most two, so $f^{4} = 1$.
\end{proof}

A loop $Q$ is \emph{automorphic} if $Inn(Q)\leq Aut(Q)$. Automorphic loops were introduced by Bruck and Paige~\cite{BruPa:56} and received attention in the recent years, with foundational papers~\cite{JedKinVoj:11},~\cite{KiKuPhVoj:12}. 
\begin{corollary}\label{cor:automorph}
Nonassociative Cayley--Dickson loops are not automorphic.
\end{corollary}
\begin{proof}
Let $Q_n$ be a Cayley--Dickson loop. For $n\leq 2$, $Q_n$ is a group and hence is automorphic. Note that $\left|Inn(Q_n)\right|=2^{2^n-2}>1344\cdot 2^{n-3}=\left|Aut(Q_n)\right|$ for $n>3$ (see Theorem~\ref{thm:autdirectproduct}). Let $n=3$, and let $i_{1},i_{2},i_{3}$ be canonical generators of $Q_n$. If $Inn(Q_n)\cap Aut(Q_n)=id$, we are done. Otherwise, let $f\in Inn(Q_n) \cap Aut(Q_n)$ be a nontrivial mapping defined by
\[
f(i_{k})=f_{k},\ \ k\in\{1,2,3\}.
\]
For every $x\in Q_n$, $x\notin \pm\{i_{1},i_{2},i_{3}\}$, we know that $x=\prod_{j=1}^3 i^{\epsilon_{j}}_{j}$ (where $\epsilon_{j}\in \{0,1\}$), and since $f$ is an automorphism, $f(x)$ is uniquely defined by
\[
f(x)=f(\prod_{j=1}^3 i^{\epsilon_{j}}_{j})=\prod_{j=1}^3 f(i^{\epsilon_{j}}_{j})=\prod_{j=1}^3 f^{\epsilon_{j}}_{j}.
\]
Let $y\notin \pm\{i_{1},i_{2},i_{3},x\}$. Then by \eqref{eqn:Tx3}, $\tilde{f}=T_{y}T_{x} f$ satisfies
\begin{eqnarray}
\nonumber \tilde{f}\upharpoonright_{\pm\{i_{1},i_{2},i_{3}\}}&=&f,\\
\nonumber \tilde{f}(x)&=&-f(x),
\end{eqnarray}
so $\tilde{f}\in Inn(Q_n)$ but $\tilde{f}\notin Aut(Q_n)$. 
\qedhere
\end{proof}
\section{Multiplication Groups}\label{sec:mult}
In this section we  establish the structure of the multiplication group of $Q_n$. We prove the auxiliary Lemmas~\ref{lemma:small},~\ref{lemma:small_product},~\ref{lemma:new-elt},~\ref{lemma:ind} and use them in the construction of Lemma~\ref{lemma:K} and the proof of Theorem~\ref{thm:semidir}. 
\begin{lemma}\label{lemma:small}
Let $G$ be a finite group, and let $g_{1},g_{2},\ldots,g_{n}$ be elements of $G$ of order $2$ such that $G=\left\langle g_{1},g_{2},\ldots ,g_{n}\right\rangle$. Then
\begin{equation*}\label{eqn:small}
\left|g_{j}g_{k}\right|=2\mbox{ iff }g_{j}g_{k}=g_{k}g_{j}, \ \ j,k\in \{1,\ldots,n\},j\neq k,
\end{equation*}
and if either holds for all $j,k$, then $G$ is an elementary abelian $2$-group.
\end{lemma}
\begin{proof}
Suppose $\left|g_{j}g_{k}\right|=2$, then $(g_{j}g_{k})(g_{k}g_{j})=g_{j}g_{k}^2g_{j}=g_{j}^2=1=(g_{j}g_{k})(g_{j}g_{k})$, and hence $g_{k}g_{j}=g_{j}g_{k}$. If $g_{k}g_{j}=g_{j}g_{k}$, then $(g_{j}g_{k})^2=(g_{j}g_{k})(g_{j}g_{k})=(g_{j}g_{k})(g_{k}g_{j})=g_{j}g^2_{k}g_{j}=g^2_{j}=1$, and $\left|g_{j}g_{k}\right|=2$. If $g_{k}g_{j}=g_{j}g_{k}$ for all $j,k\in \{1,\ldots,n\},j\neq k$, it is straightforward to check that $G$ is an elementary abelian $2$-group.
\end{proof}
\begin{lemma}\label{lemma:small_product}
Let $Q_n$ be a Cayley--Dickson loop, $i_{j},i_{k}$ among its canonical generators, and $x\in Q_n$. Let 
\begin{eqnarray*}
\nonumber p_{j,k}(x)&=&L_{i_{j}}\upharpoonright_{\pm\{x,i_{j}x,i_{k}x,i_{j}(i_{k}x)\}}\\
				&=&(x, i_{j}x, -x, -i_{j}x)(i_{k}x, i_{j}(i_{k}x), -i_{k}x, -i_{j}(i_{k}x)),\\
q_{j,k}(x)&=&T_{i_{k}x}T_{x}p_{j,k}(x),\\
M_{j,k,x,1}&=&\{T_{i_{j}x}T_{x},T_{i_{j}(i_{k}x)}T_{i_{k}x}\},\\
M_{j,k,x,-1}&=&\{T_{i_{j}(i_{k}x)}T_{x},T_{i_{k}x}T_{i_{j}x}\}.
\end{eqnarray*}
Then $\left|q_{j,k}(x)\right|=\left|t p_{k,j}(x)\right|=\left|q_{j,k}(x)(t p_{k,j}(x))\right|=2$, where $t\in M_{j,k,x,s}$, and $s\in Z(Q_n)$ satisfies $i_{j}(i_{k}x)=s(i_{k}(i_{j}x))$.
\end{lemma}
\begin{proof}
We write down the corresponding permutations and check that they only contain involutions. Using Lemma~\ref{lemma:Tx},
\[q_{j,k}(x)=T_{i_{k}x}T_{x}p_{j,k}=(x, i_{j}x)(-x, -i_{j}x)(i_{k}x,  i_{j}(i_{k}x))(-i_{k}x, -i_{j}(i_{k}x)),\]
hence $\left|q_{j,k}(x)\right|=2$.\\
Let $s$ be an element of $Q_n$ such that $i_{j}(i_{k}x)=s(i_{k}(i_{j}x))$. Note that $s\in Z(Q_n)$ as a product of commutators and associators, therefore $s\in \{1,-1\}$.\\
If $s=1$ and $i_{j}(i_{k}x)=i_{k}(i_{j}x)$, then 
\begin{eqnarray}
\nonumber p_{k,j}(x)&=&(x, i_{k}x, -x, -i_{k}x)(i_{j}x, i_{j}(i_{k}x), -i_{j}x, -i_{j}(i_{k}x)),\\
\nonumber T_{i_{j}x}T_{x}p_{k,j}(x)&=&(x, i_{k}x)(-x, -i_{k}x)(i_{j}x, i_{j}(i_{k}x))(-i_{j}x, -i_{j}(i_{k}x)),\\
\nonumber T_{i_{j}(i_{k}x)}T_{i_{k}x}p_{k,j}(x)&=&(x, -i_{k}x)(-x, i_{k}x)(i_{j}x, -i_{j}(i_{k}x))(-i_{j}x, i_{j}(i_{k}x)).
\end{eqnarray}
In this case,
\begin{eqnarray}
\nonumber q_{j,k}(x)\cdot (T_{i_{j}x}T_{x}p_{k,j}(x))&=&(x, i_{j}(i_{k}x))(-x, -i_{j}(i_{k}x))(i_{j}x, i_{k}x)(-i_{j}x, -i_{k}x),\\
\nonumber q_{j,k}(x)\cdot (T_{i_{j}(i_{k}x)}T_{i_{k}x}p_{k,j}(x))&=&(x, -i_{j}(i_{k}x))(-x, i_{j}(i_{k}x))(i_{j}x, -i_{k}x)(-i_{j}x, i_{k}x).
\end{eqnarray}
One can see that $\left|t p_{k,j}(x)\right|=\left|q_{j,k}(x)(t p_{k,j}(x))\right|=2$, where $t\in \{T_{i_{j}x}T_{x},T_{i_{j}(i_{k}x)}T_{i_{k}x}\}$.\\
Similarly, if $s=-1$ and $i_{j}(i_{k}x)=-i_{k}(i_{j}x)$, then 
\begin{eqnarray}
\nonumber p_{k,j}(x)&=&(x, i_{k}x, -x, -i_{k}x)(i_{j}x, -i_{j}(i_{k}x), -i_{j}x, i_{j}(i_{k}x)),\\
\nonumber T_{i_{j}(i_{k}x)}T_{x}p_{k,j}(x)&=&(x, i_{k}x)(-x, -i_{k}x)(i_{j}x, i_{j}(i_{k}x))(-i_{j}x, -i_{j}(i_{k}x)),\\
\nonumber T_{i_{k}x}T_{i_{j}x}p_{k,j}(x)&=&(x, -i_{k}x)(-x, i_{k}x)(i_{j}x, -i_{j}(i_{k}x))(-i_{j}x, i_{j}(i_{k}x)).
\end{eqnarray}
In this case,
\begin{eqnarray}
\nonumber q_{j,k}(x)\cdot (T_{i_{j}(i_{k}x)}T_xp_{k,j}(x))&=&(x, i_{j}(i_{k}x))(-x, -i_{j}(i_{k}x))(i_{j}x, i_{k}x)(i_{j}x, i_{k}x),\\
\nonumber q_{j,k}(x)\cdot (T_{i_{k}x}T_{i_{j}x}p_{k,j}(x))&=&(x, -i_{j}(i_{k}x))(-x, i_{j}(i_{k}x))(i_{j}x, -i_{k}x)(-i_{j}x, i_{k}x).
\end{eqnarray}
Again, $\left|t p_{k,j}(x)\right|=\left|q_{j,k}(x)(t p_{k,j}(x))\right|=2$, where $t\in \{T_{i_{j}(i_{k}x)}T_{x},T_{i_{k}x}T_{i_{j}x}\}$.
\end{proof}
We use the following property to prove Lemmas~\ref{lemma:new-elt} and~\ref{lemma:ind}.
\begin{lemma}\label{lemma:new-elt}
Let $Q_n$ be a Cayley--Dickson loop, and let $i_{1},i_{2},\ldots,i_{n}$ be its canonical generators. Then $i_{k}(i_{n}x)=-i_{n}(i_{k}x)$ for any $x\in \left\langle i_{1},i_{2},\ldots,i_{n-1}\right\rangle$, $k<n$.
\end{lemma}
\begin{proof}
Let $x\in \left\langle i_{1},i_{2},\ldots,i_{n-1}\right\rangle$. Then
\begin{eqnarray}
\nonumber i_{k}(i_{n}x)&=&[i_{k},i_{n},x](i_{k}i_{n})x=[i_{k},i_{n}][i_{k},i_{n},x](i_{n}i_{k})x\\
\nonumber &=&[i_{k},i_{n}][i_{k},i_{n},x][i_{n},i_{k},x]i_{n}(i_{k}x).
\end{eqnarray}
Recall that $\left\langle x,y,i_{n}\right\rangle\leq \mathbb{O}_{16}$ for any $x,y \in Q_n$, by Lemma~\ref{lemma:aut2}, and $\left\langle x,y,i_{n}\right\rangle\cong \mathbb{O}_{16}$ implies that $[x,y,i_{n}]=-1$. Also, $[i_{k},i_{n}]=-1$ as $\left\langle i_{k},i_{n}\right\rangle\cong\mathbb{H}_{8}$. This leads to 
\begin{equation*}
[i_{k},i_{n}][i_{k},i_{n},x][i_{n},i_{k},x]=
\begin{cases} 
-1\cdot 1\cdot 1=-1,& \mbox{ if }x\in\left\langle i_{k},i_{n}\right\rangle,\\
-1\cdot (-1)\cdot (-1)=-1,& \mbox{ otherwise.}
\end{cases} 
\end{equation*}
We conclude that
\begin{equation*}
i_{k}(i_{n}x)=-i_{n}(i_{k}x).
\qedhere
\end{equation*}
\end{proof}
\begin{lemma}\label{lemma:ind}
Let $Q_n$ be a Cayley--Dickson loop, and let $i_{1},i_{2},\ldots,i_{n}$ be its canonical generators. For any $x\in \left\langle i_{1},i_{2},\ldots,i_{n-1}\right\rangle$, $j<n,k<n,j\neq k$, if $i_{j}(i_{k}x)=s(i_{k}(i_{j}x))$, then $i_{j}(i_{k}(xi_{n}))=s(i_{k}(i_{j}(xi_{n})))$ (where $s\in Z(Q_n)$).   
\end{lemma}
\begin{proof}
Let $x\in \left\langle i_{1},i_{2},\ldots,i_{n-1}\right\rangle$, and let $s\in Z(Q_n)$ be such that $i_{j}(i_{k}x)=s(i_{k}(i_{j}x))$. Then 
\begin{eqnarray}
\nonumber i_{j}(i_{k}(xi_{n})) &=& [i_{k},x,i_{n}]i_{j}((i_{k}x)i_{n})= [i_{k},x,i_{n}][i_{j},i_{k}x,i_{n}](i_{j}(i_{k}x))i_{n}\\
\nonumber &=&[i_{k},x,i_{n}][i_{j},i_{k}x,i_{n}]s((i_{k}(i_{j}x))i_{n})\\
\nonumber &=&[i_{k},x,i_{n}][i_{j},i_{k}x,i_{n}][i_{k},i_{j}x,i_{n}]s i_{k}((i_{j}x)i_{n})\\
\nonumber &=&[i_{k},x,i_{n}][i_{j},i_{k}x,i_{n}][i_{k},i_{j}x,i_{n}][i_{j},x,i_{n}]s i_{k}(i_{j}(xi_{n})).
\end{eqnarray}
Recall that $\left\langle x,y,i_{n}\right\rangle\leq \mathbb{O}_{16}$ for any $x,y \in Q_n$, by Lemma~\ref{lemma:aut2}, and $\left\langle x,y,i_{n}\right\rangle\cong \mathbb{O}_{16}$ implies that $[x,y,i_{n}]=-1$, which leads to 
\begin{equation*}
[i_{k},x,i_{n}][i_{j},i_{k}x,i_{n}][i_{k},i_{j}x,i_{n}][i_{j},x,i_{n}]=
\begin{cases} 
1\cdot (-1)\cdot (-1)\cdot 1=1,& \mbox{ if }x=\pm 1,\\
-1\cdot (-1)\cdot 1\cdot 1=1,& \mbox{ if }x=\pm i_{j},\\
1\cdot 1\cdot (-1)\cdot (-1)=1,& \mbox{ if }x=\pm i_{k},\\
-1\cdot 1\cdot 1\cdot (-1)=1,& \mbox{ if }x=\pm i_{j}i_{k},\\
-1\cdot (-1)\cdot (-1)\cdot (-1)=1& \mbox{ otherwise.}
\end{cases} 
\end{equation*}
We conclude that
\begin{equation*}
i_{j}(i_{k}(xi_{n})) = s(i_{k}(i_{j}(xi_{n}))).
\qedhere
\end{equation*}
\end{proof}
%
%

In Lemma~\ref{lemma:K} we present a construction of the subgroup $K$ of $Mlt(Q_n)$ which is used in Theorem~\ref{thm:semidir} to establish that $Mlt(Q_n)\cong (Inn(Q_n)\times Z(Q_n))\rtimes K$. For every $x\in Q_n/ \{1,-1\}$, we want $K$ to contain the element $k_x$ such that $k_x(1)\in x$. This holds when $K$ is generated by $\{L_{i_k}\left|\right. i_k\mbox{ a canonical generator of }Q_n\}$. We also want $K$ to be sufficiently small to allow $(Inn(Q_n)\times Z(Q_n))\cap K=id$. To achieve this, we should adjust the left translations $L_{i_k}$ so that they generate a group as small as needed. This is done by multiplying $L_{i_k}$ by $\psi_k\in Inn(Q_n)$ such that $\left|\psi_k L_{i_k}\right|=\left|\psi_j L_{i_j}\right|=\left|(\psi_k L_{i_k})\cdot (\psi_j L_{i_j})\right|=2$ for all $j,k\leq n, j\neq k$. 
Consider a group $\mathbb{H}_8$, where left translations by canonical generators are
\begin{eqnarray}
 \nonumber L_{i_1}&=&(1, i_1, -1, -i_1)(i_2, i_1i_2, -i_2, -i_1i_2),\\
 \nonumber L_{i_2}&=&(1, i_2, -1, -i_2)(i_1, i_2i_1, -i_1, -i_2i_1)=(1, i_2, -1, -i_2)(i_1, -i_1i_2, -i_1, i_1i_2).
\end{eqnarray}
For an inner mapping $\psi_1\in Inn(\mathbb{H}_8)$ such that $\left|\psi_1 L_{i_1}\right|=2$ we can either take $T_{i_{2}}$, or $T_{i_{1}i_{2}}$ (one can check that $\left|T_{i_{1}}L_{i_1}\right|=4$), 
\begin{eqnarray}
 \nonumber T_{i_{2}}L_{i_1}&=&(1, -i_1)(-1, i_1)(i_2, -i_1i_2)(-i_2, i_1i_2),\\
 \nonumber T_{i_{1}i_{2}}L_{i_1}&=&(1, -i_1)(-1, i_1)(i_2, i_1i_2)(-i_2, -i_1i_2).
\end{eqnarray}
Similarly, for an inner mapping $\psi_2\in Inn(\mathbb{H}_8)$ such that $\left|\psi_2 L_{i_2}\right|=2$ we can either take $T_{i_{1}}$, or $T_{i_{1}i_{2}}$, 
\begin{eqnarray}
 \nonumber T_{i_{1}}L_{i_2}&=&(1, -i_2)(-1, i_2)(i_1, i_1i_2)(-i_1, -i_1i_2),\\
 \nonumber T_{i_{1}i_{2}}L_{i_2}&=&(1, -i_2)(-1, i_2)(i_1, -i_1i_2)(-i_1, i_1i_2).
\end{eqnarray}
For a pair of mappings $\psi_1, \psi_2$ such that $\left|(\psi_1 L_{i_1})\cdot (\psi_2 L_{i_2})\right|=2$ we can either take
$\psi_1=T_{i_{2}}, \psi_2=T_{i_{1}i_{2}}$, or $\psi_1=T_{i_{1}i_{2}}, \psi_2=T_{i_{1}}$,
\begin{eqnarray}
 \nonumber (T_{i_{2}}L_{i_1})\cdot(T_{i_{1}i_{2}}L_{i_2})&=&(1, i_1i_2)(-1, -i_1i_2)(i_1, i_2)(-i_1, -i_2),\\
 \nonumber (T_{i_{1}i_{2}}L_{i_1})\cdot(T_{i_{1}}L_{i_2})&=&(1, -i_1i_2)(-1, i_1i_2)(i_1, i_2)(-i_1, -i_2).
\end{eqnarray}
Without loss of generality, we choose $\psi_1=T_{i_{2}}, \psi_2=T_{i_{1}i_{2}}$, and $K_2=\left\langle g_{1,2},g_{2,2}\right\rangle=\left\langle T_{i_{2}}L_{i_1},T_{i_{1}i_{2}}L_{i_2}\right\rangle$. The group $K_2$ is not unique, and this particular choice allows to generalize the construction for higher dimensions. The group we present in Lemma~\ref{lemma:K} is based on this choice and suffices to establish the structure of $Mlt(Q_n)$. Note that the structure of $Mlt(\mathbb{H}_8)$ is known (see Lemma~\ref{rmk:mlt_g}), so the construction of $K$ for $\mathbb{H}_8$ is only used as an initial step of the inductive construction for $Q_n$. 

Next, consider $\mathbb{O}_{16}$. We want to construct $K_3$ based on $K_2$ by extending the generators of $K_2$ to form the elements of $K_3$, and including one more generator based on $L_{i_3}$. By Lemma~\ref{lemma:comm-assoc}, we have
\begin{eqnarray}
 \nonumber i_1(i_2i_3)&=&-(i_1i_2)i_3,\\
 \nonumber i_2(i_1i_3)&=&-(i_2i_1)i_3=(i_1i_2)i_3, \\
 \nonumber i_3(i_1i_2)&=&-(i_1i_2)i_3,
\end{eqnarray}
hence
\begin{eqnarray}
 \nonumber L_{i_1}&=&(1, i_1, -1, -i_1)(i_2, i_1i_2, -i_2, -i_1i_2)(i_3, i_1i_3, -i_3, -i_1i_3)(i_2i_3, -(i_1i_2)i_3, -i_2i_3, (i_1i_2)i_3),\\ 
 \nonumber L_{i_2}&=&(1, i_2, -1, -i_2)(i_1, -i_1i_2, -i_1, i_1i_2)(i_3, i_2i_3, -i_3, -i_2i_3)(i_1i_3, (i_1i_2)i_3, -i_1i_3, -(i_1i_2)i_3),\\ 
 \nonumber L_{i_3}&=&(1, i_3, -1, -i_3)(i_1, -i_1i_3, -i_1, i_1i_3)(i_2, -i_2i_3, -i_2, i_2i_3)(i_1i_2, -(i_1i_2)i_3, -i_1i_2, (i_1i_2)i_3).
\end{eqnarray}
For every cycle $(x, i_kx, -x, -i_kx)$ we want $\psi_k$ to include either $T_x$, or $T_{i_kx}$ (but not both), so that the cycle becomes a product of two $2$-cycles, either $(x, i_kx)(-x, -i_kx)$, or $(x, -i_kx)(-x, i_kx)$. Note that a product of an odd number of mappings $T_{x_1}T_{x_2}T_{x_3}$ (where $x_1,x_2,x_3\in \mathbb{O}_{16}$) fixes $\pm\{1,x_1,x_2,x_3\}$ and moves all other elements (see Lemma~\ref{lemma:Tx}). Taking $\psi_1=T_{i_{2}}T_{i_{3}}T_{i_{2}i_{3}}, \psi_2=T_{i_{1}i_{2}}T_{i_{3}}T_{i_{1}i_{2}i_{3}}$, we get
\begin{eqnarray}
 \nonumber g_{1,3}&=&T_{i_{2}}T_{i_{3}}T_{i_{2}i_{3}}L_{i_1}\\
 \nonumber &=&(1, -i_1)(-1, i_1)(i_2, -i_1i_2)(-i_2, i_1i_2)\\
 \nonumber &&(i_3, -i_1i_3)(-i_3, i_1i_3)(i_2i_3, (i_1i_2)i_3)(-i_2i_3, -(i_1i_2)i_3),\\
 \nonumber g_{2,3}&=&T_{i_{1}i_{2}}T_{i_{3}}T_{i_{1}i_{2}i_{3}}L_{i_2}\\
 \nonumber &=&(1, -i_2)(-1, i_2)(i_1, -i_1i_2)(-i_1, i_1i_2)\\
 \nonumber &&(i_3, -i_2i_3)(-i_3, i_2i_3)(i_1i_3, (i_1i_2)i_3)(-i_1i_3, -(i_1i_2)i_3),\\ 
 \nonumber g_{1,3}g_{2,3} &=& (1, i_1i_2)(-1, -i_1i_2)(i_1, i_2)(-i_1, -i_2)\\
 \nonumber &&(i_3, -(i_1i_2)i_3)(-i_3, (i_1i_2)i_3)(i_1i_3, i_2i_3)(-i_1i_3, -i_2i_3).
\end{eqnarray}
Again, this is one of several possible choices of $g_{1,3}, g_{2,3}$. Finally, we need to add a generator $g_{3,3}$ such that $\left|g_{3,3}\right|=\left|g_{1,3}g_{3,3}\right|=\left|g_{2,3}g_{3,3}\right|=2$, one can choose, for example,
\begin{eqnarray}
 \nonumber g_{3,3}&=&T_{i_1i_2}T_{i_1i_3}T_{i_2i_3}L_{i_3}\\
 \nonumber &=&(1, -i_3)(-1, i_3)(i_1, -i_1i_3)(-i_1, i_1i_3)\\
 \nonumber &&(i_2, -i_2i_3)(-i_2, i_2i_3)(i_1i_2, (i_1i_2)i_3)(-i_1i_2, -(i_1i_2)i_3),
\end{eqnarray}
which results in 
\begin{eqnarray}
 \nonumber g_{1,3}g_{3,3}&=&(1, i_1i_3)(-1, -i_1i_3)(i_1, i_3)(-i_1, -i_3)\\
 \nonumber &&(i_2, -(i_1i_2)i_3)(-i_2, (i_1i_2)i_3)(i_1i_2, i_2i_3)(-i_1i_2, -i_2i_3),\\
 \nonumber g_{2,3}g_{3,3}&=&(1, i_2i_3)(-1, -i_2i_3)(i_2, i_3)(-i_2, -i_3)\\
 \nonumber &&(i_1, -(i_1i_2)i_3)(-i_1, (i_1i_2)i_3)(i_1i_2, i_1i_3)(-i_1i_2, -i_1i_3).
\end{eqnarray}

Below is the description of the construction for $Q_n$.
\begin{lemma}\label{lemma:K}
Let $i_{1},i_{2},\ldots,i_{n}$ be canonical generators of a Cayley--Dickson loop $Q_n$, and let $K$ be the group constructed inductively as follows
\begin{eqnarray}
\nonumber s_{1,2}&=&\{1,i_{2}\},\ \ s_{2,2}=\{1,i_{1}i_{2}\},\\
\nonumber g_{1,2}&=&(\prod_{x\in s_{1,2}} T_x)L_{i_{1}},\\
\nonumber g_{2,2}&=&(\prod_{x\in s_{2,2}} T_x)L_{i_{2}},\\
\nonumber K_{2}&=&\left\langle g_{1,2},g_{2,2}\right\rangle,\\
\nonumber s_{k,n}&=&\{x,i_{n}x\left|\right.x\in s_{k,n-1}\},\ \ k\in\{1,\ldots,n-1\},\\
\nonumber s_{n,n}&=&\left\{\prod_{j=1}^{n}i_{j}^{p_{j}}\left|\right.p_{j}\in\{0,1\},\sum_{j=1}^{n}p_{j}\in 2\mathbb{Z}\right\},\\
\nonumber g_{k,n}&=&(\prod_{x\in s_{k,n}} T_x)L_{i_{k}}=(\prod_{x\notin s_{k,n}} (x,-x))L_{i_{k}},\ \ k\in\{1,\ldots,n\},\\
\nonumber K&=&K_{n}=\left\langle g_{1,n},g_{2,n},\ldots,g_{n,n}\right\rangle.
\end{eqnarray}
Then $K$ is an elementary abelian $2$-group of order $2^n$.
\end{lemma}
\begin{proof}
We show by induction on $n$ that generators of $K$ have order~2. If $n=2$, then 
\begin{eqnarray}
\nonumber g_{1,2}&=&T_{i_{2}}T_{1}L_{i_{1}}=(1, -i_{1})(-1, i_{1})(i_{2}, -i_{1}i_{2})(-i_{2}, i_{1}i_{2}),\\
\nonumber g_{2,2}&=&T_{i_{1}i_{2}}T_{1}L_{i_{2}}=(1, -i_{2})(-1, i_{2})(i_{1}, -i_{1}i_{2})(-i_{1}, i_{1}i_{2})
\end{eqnarray}
are of order $2$. 
Suppose that generators $g_{1,n-1},g_{2,n-1},\ldots ,g_{n-1,n-1}$ of $K_{n-1}$ have order $2$. Note that a product of an odd number of mappings $T_{x_1}\ldots T_{x_{2^{n-1}-1}}$ (where $x_1,\ldots,x_{2^{n-1}-1}\in 
Q_n$) fixes $\pm\{1,x_1,\ldots,x_{2^{n-1}-1}\}$ and moves all other elements (see Lemma~\ref{lemma:Tx}). Left translation $L_{i_{k}}$ consists of 4-cycles of the form \[(x, i_{k}x, -x, -i_{k}x).\] In order to transform such cycle into two $2$-cycles, $L_{i_{k}}$ is multiplied by either $T_x$, or $T_{i_{k}x}$. Then the cycle \[(i_{n}x, i_{k}(i_{n}x), -i_{n}x, -i_{k}(i_{n}x))\] that appears in the next step of the inductive construction, is multiplied by either $T_{i_{n}x}$ or $T_{i_{n}(i_{k}x)}$, leading to either 
\begin{eqnarray}
\nonumber &&(i_{n}x, -i_{k}(i_{n}x))(-i_{n}x, i_{k}(i_{n}x)),\mbox{ or} \\
\nonumber &&(i_{n}x, i_{k}(i_{n}x))(-i_{n}x, -i_{k}(i_{n}x)), 
\end{eqnarray}
respectively.\\
If $x=1$, the 4-cycle that corresponds to $(1, i_{k}, -1, -i_{k})$ in the next step of the inductive construction is $(i_{n}, i_{k}i_{n}, -i_{n}, -i_{k}i_{n})$, 
which is multiplied by $T_{i_{n}}$ and becomes $(i_{n}, -i_{k}i_{n})(-i_{n}, i_{k}i_{n})$. 
It follows that $g_{k,n}$ consists of $2$-cycles and therefore $\left|g_{k,n}\right|=2$. \\
Consider a generator $g_{n,n}$ added at the $n$-th step of the inductive construction. 
Left translation $L_{i_{n}}$ consists of cycles of the form 
\[(x, i_{n}x, -x, -i_{n}x)\] 
where either $x$, or $i_{n}x$ (but not both) is a product of even number of units $i_{k}$, for some $k\leq n$. 
In the former case, if $x\neq\pm 1$, then multiplication of $L_{i_{n}}$ by $T_x$ transforms a cycle 
$(x, i_{n}x, -x, -i_{n}x)$ into 
\[(x, -i_{n}x)(-x, i_{n}x),\]
otherwise, multiplication of $L_{i_{n}}$ by $T_{i_{n}x}$ transforms it into 
\[(x, i_{n}x)(-x, -i_{n}x).\]
Also, $L_{i_{n}}$ is multiplied by $T_{x_{k}}, x_{k}\neq \pm i_{n}$, an odd number of times, mapping $i_{n}$ to $-i_{n}$, and a cycle $(1, i_{n}, -1, -i_{n})$ becomes 
\[(1, -i_{n})(-1, i_{n}).\] 
Generator $g_{n,n}$ consists of $2$-cycles and therefore $\left|g_{n}\right|=2$.\\
Next, use induction on $n$ to show that 
\[\left|g_{j,n}g_{k,n}\right|=2,\mbox{ for all } j,k\in\{1,\ldots,n\},j\neq k.\] 
If $n=2$, then 
\[g_{1,2}g_{2,2}=(1, i_{1}i_{2})(-1, -i_{1}i_{2})(i_{1}, i_{2})(-i_{1}, -i_{2})\]
and $\left|g_{1,2}g_{2,2}\right|=2.$ Suppose that $\left|g_{j,n-1}g_{k,n-1}\right|=2$ for any pair of generators 
$g_{j,n-1},g_{k,n-1}$ of $K_{n-1}$. Without loss of generality, let $j<k$. Up to renaming $x$ and $i_{j}x$, the cycles
\[
p_{j,k}(x)=L_{i_{j}}\upharpoonright_{\pm\{x,i_{j}x,i_{k}x,i_{j}(i_{k}x)\}}=(x, i_{j}x, -x, -i_{j}x)(i_{k}x, i_{j}(i_{k}x), -i_{k}x, -i_{j}(i_{k}x))
\] 
are acted upon by $T_{i_{k}x}T_{x}$ to construct $g_{j,n}$. Then, by Lemma~\ref{lemma:small_product}, the cycles 
\[
p_{k,j}(x)=L_{i_{k}}\upharpoonright_{\pm\{x,i_{j}x,i_{k}x,i_{j}(i_{k}x)\}}=(x, i_{k}x, -x, -i_{k}x)(i_{j}x, i_{k}(i_{j}x), -i_{j}x, -i_{k}(i_{j}x))
\]  
are acted upon by $t\in M_{j,k,x,s}$, where $t=T_{y}T_{z}$ for some $y,z\in Q_n$. The cycles
\[p_{j,k}(xi_{n})=(xi_{n}, i_{j}(xi_{n}), -xi_{n}, -i_{j}(xi_{n}))(i_{k}(xi_{n}), i_{j}(i_{k}(xi_{n})), -i_{k}(xi_{n}), -i_{j}(i_{k}(xi_{n}))) \]
added at the next step of the inductive construction are multiplied by $T_{i_{n}(i_{k}x)}T_{i_{n}x}$. The cycles 
\begin{eqnarray}
\nonumber p_{k,j}(xi_{n})&=&L_{i_{k}}\upharpoonright_{\pm\{xi_{n},i_{j}(xi_{n}),i_{k}(xi_{n}),i_{j}(i_{k}(xi_{n}))\}}\\
\nonumber &=&(xi_{n}, i_{k}(xi_{n}), -xi_{n}, -i_{k}(xi_{n}))(i_{j}(xi_{n}), i_{k}(i_{j}(xi_{n})), -i_{j}(xi_{n}), -i_{k}(i_{j}(xi_{n})))
\end{eqnarray}
are multiplied by $T_{yi_{n}}T_{zi_{n}}\in M_{j,k,xi_{n},s}$. By Lemma~\ref{lemma:ind}, $i_{j}(i_{k}(xi_{n}))=s(i_{k}(i_{j}(xi_{n})))$, and therefore by Lemma~\ref{lemma:small_product}, 
\[
\left|(T_{(i_{k}x)i_{n}}T_{xi_{n}}p_{j,k}(xi_{n}))\cdot (t p_{k,j}(xi_{n}))\right|=2\mbox{ for }t\in M_{j,k,xi_{n},s}.
\] 
It is left to show that $\left|g_{j,n}g_{n,n}\right|=2$, where $j<n$. Up to renaming $x$ and $i_{j}x$, the cycles
\[
p_{j,k}(x)=L_{i_{j}}\upharpoonright_{\pm\{x,i_{j}x,i_{n}x,i_{j}(i_{n}x)\}}=(x, i_{j}x, -x, -i_{j}x)(i_{n}x, i_{j}(i_{n}x), -i_{n}x, -i_{j}(i_{n}x))
\] 
are acted upon by $T_{i_{n}x}T_{x}$. By Lemma~\ref{lemma:new-elt}, $i_{j}(i_{n}x)=-i_{n}(i_{j}x)$, therefore by Lemma~\ref{lemma:small_product},
\[
\left|(T_{i_{n}x}T_{x}p_{j,k}(x))\cdot (tp_{n,j}(x))\right|=2\mbox{ where }t\in \{T_{x}T_{i_{j}(i_{n}x)},T_{i_{j}x}T_{i_{n}x}\}.
\] 
If $x$ is a product of even number of units $i_{k}$, for some $k\leq n$, then $i_{j}(i_{n}x)$ is also a product of even number of units, so $x,i_{j}(i_{n}x)$ are in $s_{n,n}$, and $T_{x}T_{i_{j}(i_{n}x)}$ is a part of the construction of $g_{n,n}$. If $x$ is a product of odd number of units, then $i_{j}x,i_{n}x$ are products of even number of units, and are included in $s_{n,n}$, so $T_{i_{j}x}T_{i_{n}x}$ is a part of the construction of $g_{n,n}$. In both cases this leads to $\left|g_{j,n}g_{n,n}\right|=2$.\\
Summarizing, $K$ satisfies the assumptions of Lemma~\ref{lemma:small} and is therefore an elementary abelian $2$-group.\\
To determine the order of $K$, define a mapping $\phi:Q_n/ \{1,-1\}\to K$ by
\[
\phi(\{i_{k},-i_{k}\})=g_{k,n},k\in\{1,\ldots,n\}.
\]
Note that for any 
\begin{eqnarray}
\nonumber x&=&\pm \{\prod_{j=1}^n i^{\epsilon_j}_{j}\} \in Q_n\slash \{1,-1\}\ \ ( \mbox{ where }\epsilon_j\in \{0,1\}),\mbox{ there is} \\
\label{eqn:k-product} g&=&\prod_{j=1}^n \phi (i^{\epsilon_j}_{j})=\prod_{j=1}^n g^{\epsilon_j}_{j,n},
\end{eqnarray}
such that $g(1)\in x$. We conclude that 
\[\left|K\right|\geq \left|Q_n/ \{1,-1\}\right|=\frac{\left|Q_n\right|}{2}=2^n.\] 
Also, $K$ is an elementary abelian $2$-group with $n$ generators, so 
\begin{equation}\label{eqn:k-size}
\left|K\right|\leq 2^n.  
\end{equation}
We conclude that the order of $K$ is $2^n$.
\end{proof}
For any loop $Q$, Albert showed $Z(Mlt(Q))=\{L_x\left|\right.x\in Z(Q)\}\cong Z(Q)$. To improve legibility, we will identify $Z(Mlt(Q))$ with $Z(Q)$ in what follows.

In Theorem~\ref{thm:semidir} we use the group $N=\left\langle Inn(Q_n),Z(Q_n)\right\rangle=Inn(Q_n)Z(Q_n)$ to establish the structure of $Mlt(Q_n)$. Recall that elements of $Inn(Q_n)$ are all even products of $2$-cycles $(x, -x)$ (where $1 \neq x\in Q_n\slash \{1,-1\}$). A group $Inn(Q_n)$ stabilizes~$1$, therefore $Inn(Q_n)\cap Z(Q_n)=1$. The index $[N:Inn(Q_n)]=2$, therefore $Inn(Q_n)\trianglelefteq N$, and $Z(Q_n)\trianglelefteq Mlt(Q_n)$ implies $Z(Q_n)\trianglelefteq N$. It follows that $N=Inn(Q_n)\times Z(Q_n)$, and $N=Inn(Q_n)\cup (-Inn(Q_n))$. \\
A basis for $Inn(Q_n)$ can be taken to be
\[
\left\{ T_xT_e=(x,-x)(e,-e)\left| 1,e \neq x\in Q_n\slash \{1,-1\}\right.\right\}.
\] 
Elements of $N$ are all even products of $2$-cycles $(x, -x)$, for $x\in Q_n\slash \{1,-1\}$. A mapping $L_{-1}T_e=(1,-1)(e,-e)$ can be used to construct a basis for $N$ (see Figure~\ref{fig:innxZ2}),
\[
N^{*}=\left\{ L_{-1}T_e,T_xT_e\left| 1,e \neq x\in Q_n\slash \{1,-1\}\right.\right\}.
\] 
\begin{figure}[H]
	\centering
		\includegraphics[width=0.6\textwidth]{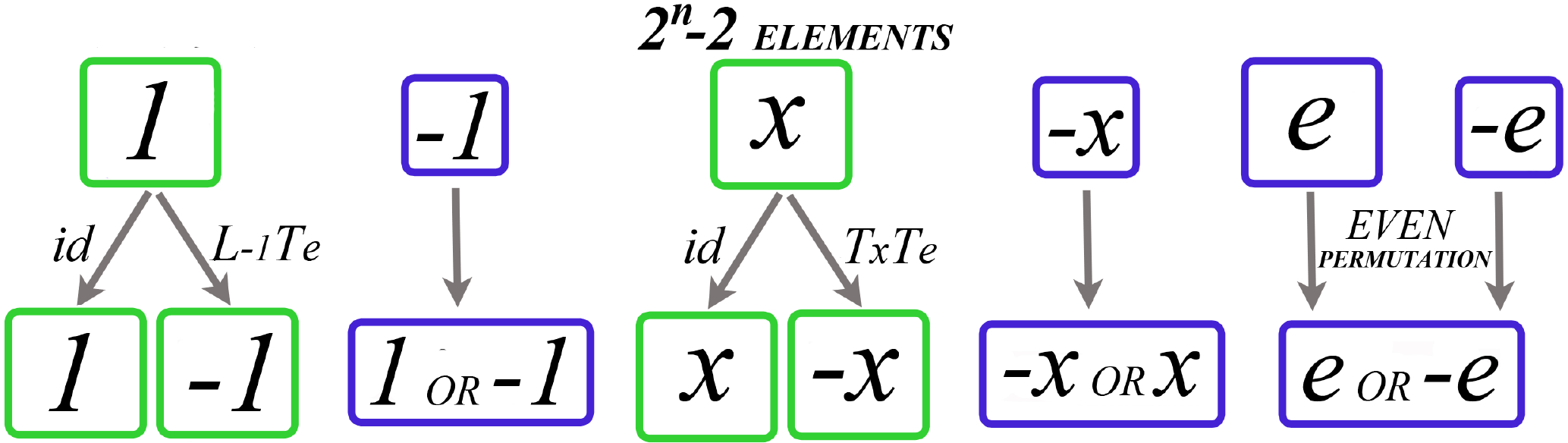}
	\caption{Group $N=Inn(Q_n)\times Z(Q_n)$}
	\label{fig:innxZ2}
\end{figure}
\begin{theorem}\label{thm:semidir}
Let $Q_n$ be a Cayley--Dickson loop, $n\geq 2$. Then $Mlt(Q_n)\cong (Inn(Q_n)\times Z(Q_n))\rtimes K$, where $K$ is the group constructed in Lemma~\ref{lemma:K}. In particular, $Mlt(Q_n)\cong ((\mathbb{Z}_{2})^{2^{n}-2}\times \mathbb{Z}_{2})\rtimes (\mathbb{Z}_{2})^{n}$.
\end{theorem}
\begin{proof}
Let $G=Mlt(Q_n)$, $N=Inn(Q_n)\times Z(Q_n)$, and $K$ be the group constructed in Lemma~\ref{lemma:K}. We want to show that $G=N\rtimes K$.
\begin{enumerate}
	\item Let $\alpha \in N, g\in G$. There exist $x\in Q_n, \beta \in Inn(Q_n)$ such that $g=\beta L_{x}$. Consider $g\alpha g^{-1}$ acting on~$1$,
	\begin{eqnarray}
	\nonumber g\alpha g^{-1}(1)&=&\beta L_{x}\alpha (\beta L_{x})^{-1}(1)=\beta L_{x}\alpha L_{x}^{-1}\underbrace{\beta^{-1}(1)}_{1}=\beta \underbrace{L_{x}\alpha L_{x}^{-1}(1)}_{\pm 1}=\pm \beta(1)=\pm 1.
	\end{eqnarray}
	This shows that $g\alpha g^{-1}\in Inn(Q_n)\cup (-Inn(Q_n))=N $, so $N$ is normal in $G$.
	\item By \eqref{eqn:k-product}, \eqref{eqn:k-size}, $K$ contains a unique element $g$ such that $g(1)\in \{1,-1\}$. Since $K$ is a group, $g=id$, thus $N \cap K=id$.
	\item We established that $N\unlhd G, K\leq G$, and $N\cap K=id$. We have $N\rtimes K\leq G$. Recall that 
	\begin{eqnarray}
	\nonumber \left[Mlt(Q_n):Inn(Q_n)\right] &=& \left|Q_n\right|\mbox{, thus} \\
	\nonumber \left[Mlt(Q_n):(Inn(Q_n)\times Z(Q_n))\right]&=&\left[Mlt(Q_n):Inn(Q_n)\right]\slash 2=2^n=\left|K\right|,
	\end{eqnarray}
	and	$(Inn(Q_n)\times Z(Q_n))\rtimes K \cong Mlt(Q_n)$ follows.
	\qedhere
\end{enumerate}
\end{proof}

%

We have shown that $Mlt(Q_n)$ is a semidirect product of two permutation groups $N$, $K$, both elementary abelian $2$-groups. Recall that if $N$, $K$ are groups and $\phi : K\to Aut(N)$ is a homomorphism, then the external semidirect product is defined on $N\times K$ by  
	\[
	(h_1,k_1)\circ (h_2,k_2)=(h_1\ast \phi_{k_1}(h_2),k_1\cdot k_2),\ \ h_1,h_2\in N,k_1,k_2\in K. 
	\]
In an internal semidirect product $G = N\rtimes_{\phi} K$, the action $\phi : K\to Aut(N)$ is natural, that is, by conjugation $\phi_{k_1}(h_2)=k_1h_2k_1^{-1}$.

\begin{lemma}\label{lemma:nat-act}\cite[p.170]{Rotman:94}
Let $G,N,K$ be finite groups such that $N\trianglelefteq G$, $K\leq G$, and $G = N\rtimes_{\phi} K$. Then $K$ acts on $N$ by conjugation.
\end{lemma}

In~\cite{Kirshtein:12-2} we construct an isomorphic copy of $Mlt(Q_n)$ as an external semidirect product of two abstract elementary abelian $2$-groups.

\section{Left and Right Inner Mapping Groups}\label{sec:rlinn}
In this section we discuss one-sided inner mapping groups of Cayley--Dickson loops.
 It is well known that $Mlt_l(Q)\cong Mlt_r(Q)$ and $Inn_l(Q)\cong Inn_r(Q)$ in any inverse property loop~$Q$. We give the proofs in Theorem~\ref{lemma:IP_RLMlt} and Corollary~\ref{cor:IP_RLInn} for completeness. 
\begin{theorem}\label{lemma:IP_RLMlt}
Let $Q$ be an inverse property loop. Then $Mlt_l(Q)\cong Mlt_r(Q)$.
\end{theorem}
\begin{proof}
Define a partial mapping $f:Mlt_l(Q)\to Mlt_r(Q)$ by $f(L_a)=R_a^{-1}$. We want to extend this mapping to a homomorphism.
Let $S\in Mlt_l(Q)$, then 
\[
S=\prod_{i=1}^nL_{a_i}^{\epsilon_i},\ \ a_i\in Q, \epsilon_i\in \{0,1\}.
\]
To verify that a mapping 
\[
f(S)=f(\prod_{i=1}^nL_{a_i}^{\epsilon_i})=\prod_{i=1}^nR_{a_i}^{-\epsilon_i}
\]
is well-defined, we show that if 
\[
S=\prod_{i=1}^nL_{a_i}^{\epsilon_i}=\prod_{j=1}^mL_{b_j}^{\phi_j}, 
\]
then
\[
f(S)=\prod_{i=1}^nR_{a_i}^{-\epsilon_i}=\prod_{j=1}^mR_{b_j}^{-\phi_j}. 
\]
Let $x\in Q$, then 
\begin{eqnarray}
\nonumber a_n^{\epsilon_n}(\ldots a_2^{\epsilon_2}(a_1^{\epsilon_1}x))&=&b_m^{\phi_m}(\ldots b_2^{\phi_2}(b_1^{\phi_1}x)),\mbox{ and}\\
\nonumber (a_n^{\epsilon_n}(\ldots a_2^{\epsilon_2}(a_1^{\epsilon_1}x)))^{-1}&=&(b_m^{\phi_m}(\ldots b_2^{\phi_2}(b_1^{\phi_1}x)))^{-1}.
\end{eqnarray}
Inverse property implies that $(xy)^{-1}=y^{-1}x^{-1}$,
thus
\begin{eqnarray}
\nonumber (a_n^{\epsilon_n}(\ldots a_2^{\epsilon_2}(a_1^{\epsilon_1}x)))^{-1}&=&((x^{-1}a_1^{-\epsilon_1})a_2^{-\epsilon_2})\ldots a_n^{-\epsilon_n}=\prod_{i=1}^nR_{a_i}^{-\epsilon_i}(x),
\end{eqnarray}
and
\begin{eqnarray}
\nonumber (b_m^{\phi_m}(\ldots b_2^{\phi_2}(b_1^{\phi_1}x)))^{-1}&=&((x^{-1}b_1^{-\phi_1})b_2^{-\phi_2})\ldots b_m^{-\phi_m}=\prod_{j=1}^mR_{b_j}^{-\phi_j}(x).
\end{eqnarray}
We conclude that 
\[
\prod_{i=1}^nR_{a_i}^{-\epsilon_i}=\prod_{j=1}^mR_{b_j}^{-\phi_j}
\]
and the mapping $f$ is well-defined. The mapping $f$ is a homomorphism since it has an inverse $g:Mlt_r(Q)\to Mlt_l(Q)$ defined by $g(R_a)=L_a^{-1}$.
\end{proof}
\begin{corollary}\label{cor:CD_RLMlt}
Let $Q_n$ be a Cayley--Dickson loop. Then $Mlt_l(Q_n)\cong Mlt_r(Q_n)$.
\end{corollary}
\begin{corollary}\label{cor:IP_RLInn}
Let $Q$ be an inverse property loop. Then $Inn_l(Q)\cong Inn_r(Q)$.
\end{corollary}
\begin{proof}
Note that $f\upharpoonright_{Inn_l(Q)}$ is an isomorphism from $Inn_l(Q)$ to $Inn_r(Q)$.
If 
\begin{eqnarray}
\nonumber S&=&\prod_{i=1}^nL_{a_i}^{\epsilon_i}\in Inn_l(Q),\mbox{ then}\\
\nonumber S(1)&=&a_n^{\epsilon_n}(\ldots a_2^{\epsilon_2}(a_1^{\epsilon_1}1))=1.
\end{eqnarray}
Taking the inverse, we have 
\[
(a_n^{\epsilon_n}(\ldots a_2^{\epsilon_2}(a_1^{\epsilon_1}1)))^{-1}=(((1a_1^{-\epsilon_1})a_2^{-\epsilon_2})\ldots a_n^{-\epsilon_n})=1,
\]
thus
\[
f(Inn_l(Q))=\prod_{i=1}^nR_{a_i}^{-\epsilon_i}\in Inn_r(Q).
\qedhere
\]
\end{proof}
In fact, a stronger statement holds for the Cayley--Dickson loops. As can be seen in the following Lemma, when $Q_n$ is a Cayley--Dickson loop, the left inner mapping groups $Inn_l(Q_n)$ are equal to the right inner mapping groups $Inn_r(Q_n)$.
\begin{lemma}\label{cor:RL}
Let $Q_n$ be a Cayley--Dickson loop. Then $Inn_l(Q_n)=Inn_r(Q_n)$, and $Inn(Q_n)=\left\langle T_{x},L_{x,y}\left|\right.x,y\in Q_n\right\rangle$.
\end{lemma}
\begin{proof}
For all $x,y\in Q_n$, we have $L_{x,y}=R_{x,y}$ by Corollary~\ref{cor:RLmaps}.
\end{proof}
Lemma~\ref{lemma:ik-assoc} serves a purpose similar to that of Lemma~\ref{lemma:aut2}, providing information about associators. Lemmas~\ref{lemma:ik-assoc}~and~\ref{lemma:r_prod} are used in the proof of Theorem~\ref{thm:rlinn}. 
\begin{lemma}\label{lemma:ik-assoc}
Let $Q_n$ be a Cayley--Dickson loop, $i_k$ its canonical generator, $x\in Q_k$, $y\in Q_ke$, $n\geq 4$, and $k<n$. Then
\begin{equation*}
[i_k,x,y]=[x,i_k,y]=
\begin{cases}
1, & \mbox{when } y \in Q_ke \backslash\pm\{e,i_ke,xe,xi_ke\},\\
-1, & \mbox{otherwise}.
\end{cases}
\end{equation*}
Moreover, if $x\notin \pm\{1,i_k\}$, then
\begin{equation*}
\left\langle i_k,x,y\right\rangle\cong
\begin{cases}
\tilde{\mathbb{O}}_{16}, & \mbox{when } y \in Q_ke \backslash\pm\{e,i_ke,xe,xi_ke\},\\
\mathbb{O}_{16}, & \mbox{otherwise}.
\end{cases}
\end{equation*}
\end{lemma}
\begin{proof}
Since $x\in Q_k$ and $y \in Q_ke$, we get $y\notin \left\langle i_k,x\right\rangle$. Consider the loop $\left\langle i_k,x,y\right\rangle$. If $x\in \pm\{1,i_k\}$, then $\left\langle i_k,x,y\right\rangle \cong \mathbb{H}_8$ and $[i_k,x,y]=[x,i_k,y]=1$. If $x\notin \pm\{1,i_k\}$, then $\left\langle i_k,x\right\rangle \cong \mathbb{H}_8$ and $\left|\left\langle i_k,x,y\right\rangle\right|=16$ by Lemma~\ref{lemma:order}. In this case, if $y \in \pm\{e,i_ke,xe,xi_ke\}$, then $\left\langle i_k,x,y\right\rangle =\pm\{1,x,i_k,xi_k,e,xe,i_ke,xi_ke\} \cong \mathbb{O}_{16}$ and $[i_k,x,y]=[x,i_k,y]=-1$ by Lemmas~\ref{lemma:aut2}~and~\ref{lemma:comm-assoc}. It remains to consider $x\in Q_k\backslash\pm\{1,i_k\}$, $y \in Q_ke\backslash\pm\{e,i_ke,xe,xi_ke\}$. We can write $y=ze$ for some $z\in Q_k$, then
\begin{eqnarray*}
(i_kx)(ze)&=&[i_kx,z,e]((i_kx)z)e=[i_kx,z,e][i_k,x,z](i_k(xz))e\\
					&=&[i_kx,z,e][i_k,x,z][i_k,xz,e]i_k((xz)e)\\
					&=&[i_kx,z,e][i_k,x,z][i_k,xz,e][x,z,e]i_k(x(ze))\\
					&=&i_k(x(ze)),\\
(xi_k)(ze)&=&[xi_k,z,e]((xi_k)z)e=[xi_k,z,e][x,i_k,z](x(i_kz))e\\
					&=&[xi_k,z,e][x,i_k,z][x,i_kz,e]x((i_kz)e)\\
					&=&[xi_k,z,e][x,i_k,z][x,i_kz,e][i_k,z,e]x(i_k(ze))=x(i_k(ze)),
\end{eqnarray*}
since $x,z\in Q_k$, and 
\begin{eqnarray}
\nonumber [i_kx,z,e]&=&[i_k,x,z]=[i_k,xz,e]=[x,z,e]=-1,\\
\nonumber [xi_k,z,e]&=&[x,i_k,z]=[x,i_kz,e]=[i_k,z,e]=-1 
\end{eqnarray}
by Lemmas~\ref{lemma:aut2}~and~\ref{lemma:comm-assoc}. Thus 
\begin{eqnarray}
\nonumber [i_k,x,ze]&=&[i_k,x,y]=1,\\
\nonumber [x,i_k,ze]&=&[x,i_k,y]=1.
\end{eqnarray}
If $\left|\left\langle i_k,x,y\right\rangle\right|=16$ and $[i_k,x,y]=1$, then $\left\langle i_k,x,y\right\rangle\cong \tilde{\mathbb{O}}_{16}$ by Lemmas~\ref{lemma:aut2}~and~\ref{lemma:quasioct}.
\end{proof}

\begin{lemma} \cite{Kirshtein:12} \label{rmk:assoc} If $x,y,z\in Q_{n-1}$, then in $Q_n$ we have
\begin{enumerate}[(a)]
	\item $[(x,0),(y,0),(z,1)]=[x,y][z,y,x],$	
	\item $[(x,0),(y,1),(z,0)]=[x,z][y,x,z][y,z,x],$	
	\item $[(x,0),(y,1),(z,1)]=[x,y][x,z][z,x,y][x,z,y],$
  \item $[(x,1),(y,0),(z,0)]=[y,z][x,y,z],$
	\item $[(x,1),(y,0),(z,1)]=[y,x][y,z][z,y,x],$
	\item $[(x,1),(y,1),(z,0)]=[z,x][z,y][y,x,z][y,z,x],$
	\item $[(x,1),(y,1),(z,1)]=[x,y][x,z][y,z][z,x,y][x,z,y].$
\end{enumerate}
\end{lemma}

\begin{lemma}\label{cor:xe}
Let $Q_n$ be a Cayley--Dickson loop, and let $x,y\in Q_n$ such that $x=(\bar{x},x_n)$, $y=(\bar{y},y_n)$, $\bar{x},\bar{y}\in Q_{n-1}$, $x_n, y_n \in \{0,1\}$. Then 
\begin{eqnarray}
\label{eqn:xyz-ze} L_{x,y}(z)&=&[\bar{x},\bar{y}]L_{x,y}(ze),\\	 
\nonumber L_{x,e}&=&L_{xe,e}.
\end{eqnarray}
\end{lemma}
\begin{proof}
By~Lemma~\ref{1}, $[x,y,z]=[z,y,x]$ for any $x,y,z \in Q_{n-1}$. Thus \eqref{eqn:xyz-ze} follows from Lemma~\ref{rmk:assoc}.

%
If $x\in\pm \{1,e\}$, then $L_{x,e}=L_{xe,e}=id$. Otherwise,
\begin{equation*}
L_{x,e}(z)=[e,x,z]z=
\begin{cases}
z, &\mbox{if }z\in \pm \{1,x,e,xe\},\\
-z &\mbox{otherwise,}
\end{cases}
\end{equation*}
and
\begin{equation*}
L_{xe,e}(z)=[e,xe,z]z=
\begin{cases}
z, &\mbox{if }z\in \pm \{1,x,e,xe\},\\
-z &\mbox{otherwise,}
\end{cases}
\end{equation*}
thus $L_{x,e}=L_{xe,e}$.
\end{proof}
\begin{lemma}\label{lemma:r_prod}
Let $Q_n$ be a Cayley--Dickson loop, $n\geq 4$. Then an inner mapping on $Q_n$ 
\[
h=\prod_{x\in Q_{n-2}\slash\{1,-1\}}L_{x,i_{n-1}}
\]
can be written as the following permutation 
\[
h=\prod_{z\in (Q_{n}\slash \{1,-1\})\backslash (Q_{n-1}\slash \{1,-1\})}(z,-z).
\]
\end{lemma}
\begin{proof}
Let $x\in Q_{n-2}\slash\{1,-1\}$. By~\eqref{gen2},
\[
L_{x,i_{n-1}}(z)=[i_{n-1},x,z]z.
\]
If $z\in \left\langle x,i_{n-1}\right\rangle=\pm \{1,x,i_{n-1},xi_{n-1}\}$, then $[i_{n-1},x,z]=1$. If $z\in Q_{n-1}\backslash\pm \{1,x,i_{n-1},xi_{n-1}\}$, then $[i_{n-1},x,z]=-1$ by Lemmas~\ref{lemma:aut2}~and~\ref{lemma:comm-assoc}. If $z \in \{e,xe,i_{n-1}e,xi_{n-1}e\}$, then
\[
\left\langle i_{n-1},x,z \right\rangle=\{1,x,i_{n-1},xi_{n-1},e,xe,i_{n-1}e,xi_{n-1}e\}\cong \mathbb{O}_{16}
\]
and $[i_{n-1},x,z]=-1$ by Lemmas~\ref{lemma:aut2}~and~\ref{lemma:comm-assoc}. If $z \in Q_{n-1}e\backslash\{e,xe,i_{n-1}e,xi_{n-1}e\}$, then $[i_{n-1},x,z]=1$ by Lemma~\ref{lemma:ik-assoc}. Summarizing, we have
\begin{equation*}
L_{x,i_{n-1}}(z)=
\begin{cases}
z, & \mbox{when } z \in \pm \{1,x,i_{n-1},xi_{n-1}\} \cup ( Q_{n-1}e\backslash\{e,xe,i_{n-1}e,xi_{n-1}e\}),\\
-z, & \mbox{otherwise}.
\end{cases}
\end{equation*}
Next, consider a mapping 
\[
h=\prod_{x\in Q_{n-2}\slash\{1,-1\}}L_{x,i_{n-1}}.
\]
If $z\in \pm\{1,i_{n-1}\}$, then clearly $h(z)=z$. If $z\in Q_{n-2}\backslash\pm\{1\}$, then $L_{x,i_{n-1}}(z)=-z$ for all $x\neq \pm z$, there is an even number (in fact, $2^{n-2}-2$) of such mappings, and therefore $h(z)=z$. If $z\in Q_{n-2}i_{n-1}\backslash\pm \{i_{n-1}\}$, then $z=yi_{n-1}$ for some $y\in Q_{n-2}$, and $L_{x,i_{n-1}}(z)=-z$ for all $x\neq \pm y$, there is $2^{n-2}-2$ such mappings, and therefore $h(z)=z$. We get $h(z)=z$ for $z\in Q_{n-1}$. Consider $z\in Q_{n-1}e$. If $z\in \pm\{e,i_{n-1}e\}$, then $L_{x,i_{n-1}}(z)=-z$ for all $x \neq 1$, there is $2^{n-2}-1$ such mappings, and thus $h(z)=-z$. Finally, if $z\in Q_{n-1}e\backslash\{e,i_{n-1}e\}$, then either $z=ye$, or $z=yi_{n-1}e$ for some $y\in Q_{n-2}$, and $L_{x,i_{n-1}}(z)=-z$ only when $x=\pm y$, again, $h(z)=-z$. We get $h(z)=-z$ for $z\in Q_{n-1}e$.
\end{proof}
\begin{theorem}\label{thm:rlinn}
Let $Q_n$ be a Cayley--Dickson loop. Then $Inn_l(Q_n)$ is an elementary abelian $2$-group of order $2^{2^{n-1}-1}$. 
\end{theorem}
\begin{proof}
Let $x\in Q_{n-1}\slash\{1,-1\}$, $x\neq 1$. Then by Lemma~\ref{lemma:Tx}
\[
L_{x,e}L_{i_{n-1},e}=(x,-x)(i_{n-1},-i_{n-1})(xe,-xe)(i_{n-1}e,-i_{n-1}e).
\]
For every $f\in Inn_l(Q_n)$, there is $\tilde{f}=L_{x,e}L_{i_{n-1},e}f\in Inn_l(Q_n)$ such that 
\begin{equation*}
\tilde{f}(z)=
\begin{cases}
-f(z), &\mbox{when }z\in \{x,i_{n-1},xe,i_{n-1}e\},\\
f(z), &\mbox{otherwise}.
\end{cases}
\end{equation*}
There are $2^{n-1}-2$ distinct inner mappings $L_{x,e}L_{i_{n-1},e},\ \ x\in Q_{n-1}\slash\{1,-1\},\ \ x\neq 1$, they generate a group of order $2^{2^{n-1}-2}$. 
Let  
\[
h=\prod_{y\in Q_{n-2}\slash\{1,-1\}}L_{y,i_{n-1}}=\prod_{z\in (Q_{n}\slash \{1,-1\})\backslash (Q_{n-1}\slash \{1,-1\})}(z,-z).
\]
be the mapping constructed in Lemma~\ref{lemma:r_prod}. For every $f\in Inn_l(Q_n)$, a mapping $\tilde{f}=hf$ satisfies 
\begin{equation*}
\tilde{f}(z)=
\begin{cases}
f(z), &\mbox{when }z\in Q_{n-1},\\
-f(z), &\mbox{otherwise}.
\end{cases}
\end{equation*}
The group 
\[
G=\left\langle L_{x,e}L_{i_{n-1},e}, h \left|\right. 1 \neq x\in Q_{n-1}\slash\{1,-1\}\right\rangle
\]
therefore has order $2^{2^{n-1}-1}$ and is a subgroup of $Inn_l(Q_n)$. \\
To show that $Inn_l(Q_n)=G$, recall that $L_{x,y}(z)=[\bar{x},\bar{y}]L_{x,y}(ze)$ for $\bar{x},\bar{y}\in Q_{n-1}$, by~\eqref{eqn:xyz-ze}. The value of $L_{x,y}(ze)$ is therefore uniquely determined by that of $L_{x,y}(z)$, moreover, $L_{x,y}(1)=1$, thus $Inn_l(Q_n)$ has order at most $2^{\frac{\left|Q_{n-1}\right|}{2}-1}=2^{2^{n-1}-1}$.
\end{proof}
\section{Left and Right Multiplication Groups}\label{sec:rlmlt}
Let $Q_n$ be a Cayley--Dickson loop. A group $Mlt_l(Q_n)$ is a proper subgroup of $Mlt(Q_n)$ by Theorems~\ref{thm:inn-invol}~and~\ref{thm:rlinn},
\[
Mlt_l(Q_n)_{1}=Inn_l(Q_n)<Inn(Q_n)=Mlt(Q_n)_{1}.
\]
We showed in Corollary~\ref{cor:CD_RLMlt} that $Mlt_l(Q_n)\cong Mlt_r(Q_n)$. The following theorem establishes the structure of $Mlt_l(Q_n)$.
\begin{theorem}\label{thm:rlsemidir}
Let $Q_n$ be a Cayley--Dickson loop, $n\geq 2$. Then $Mlt_l(Q_n)\cong (Inn_l(Q_n)\times Z(Q_n))\rtimes K$, where $K$ is the group constructed in Lemma~\ref{lemma:K}. In particular, $Mlt_l(Q_n)\cong ((\mathbb{Z}_{2})^{2^{n-1}-1}\times \mathbb{Z}_{2})\rtimes (\mathbb{Z}_{2})^{n}$.  
\end{theorem}
\begin{proof}
Since $Z(Q_n)\leq Mlt_l(Q_n)$, let $N=\left\langle Inn_l(Q_n),Z(Q_n)\right\rangle=Inn_l(Q_n)Z(Q_n)$. A group $Inn_l(Q_n)$ stabilizes~$1$, therefore $Inn_l(Q_n)\cap Z(Q_n)=1$. The index $[N:Inn_l(Q_n)]=2$, therefore $Inn_l(Q_n)\trianglelefteq N$, and $Z(Q_n)\trianglelefteq Mlt_l(Q_n)$ implies $Z(Q_n)\trianglelefteq N$. It follows that $N=Inn_l(Q_n)\times Z(Q_n)$. 
Let $G=Mlt_l(Q_n)$ and $K$ be the group constructed in Lemma~\ref{lemma:K}. We want to show that $G=N\rtimes K$.
\begin{enumerate}
	\item Let $\alpha \in N, g\in G$. There exist $x\in Q_n, \beta \in Inn_l(Q_n)$ such that $g=\beta L_{x}$. Consider $g\alpha g^{-1}$ acting on~$1$,
	\begin{eqnarray}
	\nonumber g\alpha g^{-1}(1)&=&\beta L_{x}\alpha (\beta L_{x})^{-1}(1)=\beta L_{x}\alpha L_{x}^{-1} \underbrace{\beta^{-1}(1)}_{1}\\
	\nonumber &=&\beta \underbrace{L_{x} \alpha L^{-1}_{x} (1)}_{\pm 1}=\pm \beta(1)=\pm 1.
	\end{eqnarray}
	This shows that $g\alpha g^{-1}\in Inn_l(Q_n)\cup (-Inn_l(Q_n))=N $, so $N$ is normal in $G$.\\
	Recall a mapping $h$ constructed in Lemma~\ref{lemma:r_prod},
	\[
	h=\prod_{z\in (Q_{n}\slash \{1,-1\})\backslash (Q_{n-1}\slash \{1,-1\})}(z,-z).
	\]
	Note that by Lemma~\ref{lemma:Tx} we have
	\[
	T_xT_{xe}T_e=\prod_{1,e,x,xe \neq z \in Q_{n}\slash \{1,-1\}}(z,-z)=L_{x,e},
	\]
	which allows to rewrite the construction in Lemma~\ref{lemma:K} as follows 
	\begin{eqnarray}
\nonumber s_{1,2}&=&\{1,i_{2}\},\ \ s_{2,2}=\{1,i_{1}i_{2}\},\\
\nonumber s_{k,n}&=&\{x,i_{n}x\left|\right.x\in s_{k,n-1}\},\ \ k\in\{1,\ldots,n-1\},\\
\nonumber s_{n,n}&=&\left\{\prod_{j=1}^{n}i_{j}^{p_{j}}\left|\right.p_{j}\in\{0,1\},\sum_{j=1}^{n}p_{j}\in 2\mathbb{Z}\right\},\\
\nonumber \bar{s}_{n,n}&=&\left\{\prod_{j=1}^{n}i_{j}^{p_{j}}\left|\right.p_{j}\in\{0,1\},\sum_{j=1}^{n}p_{j}\notin 2\mathbb{Z}\right\},\\
\nonumber 	g_{k,n}&=&(\prod_{x\in s_{k,n}} T_x)L_{i_{k}}= (\prod_{x\in s_{k,n-1}} T_xT_{xe}) T_eL_{i_{k}}\\
\nonumber				 &=&(\prod_{x\in s_{k,n-1}} T_xT_{xe}) (\prod_{x\in \{1,\ldots, 2^{n-2}-1\}}T_e)L_{i_{k}}\\
\nonumber				 &=&(\prod_{x\in s_{k,n-1}} T_xT_{xe}T_e)L_{i_{k}}\\
\nonumber				 &=&(\prod_{x\in s_{k,n-1}} L_{x,e})L_{i_{k}},\ \ k\in\{1,\ldots,n-1\},
\end{eqnarray}
\begin{eqnarray}
\nonumber 	g_{n,n}&=&(\prod_{x\in s_{n,n}} T_x)L_{i_{k}}=(\prod_{x\in s_{n-1,n-1}} T_x \prod_{x\in \bar{s}_{n-1,n-1}} T_{xe})L_{i_{k}}\\
\nonumber		 		&=&(\prod_{x\in \bar{s}_{n,n}} (x,-x))L_{i_{k}}\\
\nonumber		 		&=&(\prod_{x\in \bar{s}_{n-1,n-1}} (x,-x))(\prod_{x\in s_{n-1,n-1}} (xe,-xe))L_{i_{k}}\\
\nonumber		 		&=&(\prod_{x\in \bar{s}_{n-1,n-1}} (x,-x)(xe,-xe))(\prod_{x\in Q_{n-1}e} (x,-x))L_{i_{k}}\\
\nonumber				&=& (\prod_{x\in s_{n-1,n-1}} L_{x,e})hL_{i_{k}},\\
\nonumber 	K&=&K_{n}=\left\langle g_{1,n},g_{2,n},\ldots,g_{n,n}\right\rangle.
\end{eqnarray}
Thus $K\leq Mlt_l(Q_n)$. 
	\item By \eqref{eqn:k-product}, \eqref{eqn:k-size}, $K$ contains a unique element $g$ such that $g(1)\in \{1,-1\}$. Since $K$ is a group, $g=id$, thus $N \cap K=id$.
	\item We established that $N\unlhd G, K\leq G$, and $N\cap K=id$, therefore $N\rtimes K \leq G$. Recall that 
	\begin{eqnarray}
	\nonumber \left[Mlt_l(Q_n):Inn_l(Q_n)\right] &=& \left|Q_n\right|\mbox{, thus} \\
	\nonumber \left[Mlt_l(Q_n):(Inn_l(Q_n)\times Z(Q_n))\right]&=&\left[Mlt_l(Q_n):Inn_l(Q_n)\right]\slash 2=2^n=\left|K\right|,
	\end{eqnarray}
	and	$(Inn_l(Q_n)\times Z(Q_n))\rtimes K \cong Mlt_l(Q_n)$ follows.
	\qedhere
\end{enumerate}
\end{proof}
\paragraph{Acknowledgement} The author thanks Petr~Vojt\v{e}chovsk\'{y} for many productive discussions and for suggesting better proofs for Lemmas~\ref{lemma:mlt-even},~\ref{lemma:inn-cycles}, and Theorem~\ref{thm:inn-invol}. 

\newpage
\bibliography{cdlbib}
\bibliographystyle{plain} 
\end{document}